\documentclass[smallextended,referee,envcountsect,]{svjour3}
\smartqed
\usepackage{graphicx}
\journalname{}

\usepackage{amssymb,amsmath}
\usepackage{mathabx}
\usepackage{graphicx,multirow,amsmath,amsfonts,mathrsfs,amssymb}

\usepackage{subfigure}

\usepackage{bm,color,xcolor}
\usepackage{color}
\usepackage{verbatim}

\usepackage[ruled,linesnumbered]{algorithm2e}

 \allowdisplaybreaks [4]

\begin{document}
\title{A dynamical neural network approach for distributionally robust chance constrained Markov decision process %\footnote{} 
}
 % \author{Tian Xia$^a$, Jia Liu$^a$, Abdel Lisser$^b$\\
 % \small $a$: School of Mathematics and Statistics, Xi'an Jiaotong University, 710049, Xi'an, P. R. China\\\small 
 % $b$: CentraleSupelec, Laboratoire des Signaux et des Systemes, 91190 Gif-sur-Yvette, France}

\date{}

\author{Tian Xia, Jia Liu, Zhiping Chen }

\institute{Tian Xia \at
             School of Mathematics and Statistics, Xi'an Jiaotong University,
              Xi'an, 710049, P. R. China\\
            xt990221@stu.xjtu.edu.cn
           \and
              Jia Liu,  \at
              School of Mathematics and Statistics, Xi'an Jiaotong University,
              Xi'an, 710049, P. R. China\\
              jialiu@xjtu.edu.cn
           \and
              Zhiping Chen \at
             School of Mathematics and Statistics, Xi'an Jiaotong University, 
              Xi'an, 710049, P. R. China\\
              zchen@mail.xjtu.edu.cn
           % This research was supported by National Key R\&D Program of China under No. 2022YFA1004000 and National Natural Science Foundation of China under Grant Number 11991023 and 11901449.
}

\maketitle

\abstract{
In this paper, we study the distributionally robust joint chance constrained Markov decision process.  {Utilizing the logarithmic transformation technique,} we derive its deterministic reformulation with bi-convex terms under the moment-based uncertainty set.{To cope with the non-convexity and improve the robustness of the solution,
} we propose a dynamical neural network
approach to 
solve 
the reformulated optimization problem.  
Numerical results on a machine replacement problem demonstrate the efficiency of the proposed dynamical neural network approach when compared with the sequential convex approximation approach.

% and compare it with the Sequential convex approximation (SCA) algorithm. In the numerical experiment section, we apply the DNN approach and SCA algorithm in solving a Machine replacement problem. We analyze the pros and cons for these two approaches.
%Keywords: Markov decision process, Chance constraints, Distributionally robust optimization, Moment-based ambiguity set, Dynamical neural network. 
%Sequential convex approximation algorithm.
}


%\MSC{90C15, 90C40, 68T07}

\maketitle

\section{Introduction}

%\subsection{backgroud}
Markov decision process (MDP) is an important mathematical framework used to find an optimal dynamic policy in an uncertain environment. The randomness in an MDP often comes from two perspectives: reward \cite{varagapriya2022constrained} and transition probabilities \cite{mannor2016robust,ramani2022robust}. MDPs have been widely used in various fields, including autonomous driving \cite{brechtel2014probabilistic}, portfolio selection \cite{liu2023continual}, inventory control \cite{klabjan2013robust}, power systems \cite{wang2020mdp} and so on.

%\subsection{Literature review on constrained MDP models}
%20 sen (15-25 papers: MDP-CMP-CCMDP-DROCCMDP(DROCC)--NN)
%\subsubsection{constrained MDP}
For many applications with strong safety requirements (healthcare \cite{goyal2022robust}, autonomous vehicles \cite{you2019advanced} etc), one needs to adopt the MDP model with safe constraints. 
%As we know, the random parameters arising from uncertain events exist everywhere. When facing uncertainty of random parameters in the constraints,
 Many variations of constrained MDP have been extensively investigated, for instance, robust constrained MDP \cite{delage2010percentile}, risk constrained MDP with semi-variance \cite{yu2022zero}, Value-at-Risk \cite{ma2019state}, or Conditional Value-at-Risk \cite{prashanth2014policy}, etc.

%\subsubsection{chance constrained MDP}
As a typical extension, chance constraints are used in MDPs to ensure a high probability of constraints being satisfied.
%violated are minimized while optimizing the objective function. 
They are important in situations with high safety requirements where the violation of a random constraint might bring severe outcome \cite{kuccukyavuz2022chance}. Delage and Mannor \cite{delage2010percentile} studied reformulations of %joint 
chance constrained MDP (CCMDP) with random rewards or transition probabilities. Recently Varagapriya et al. \cite{varagapriya2022joint} applied joint chance constraints in constrained MDP  and find its reformulations when the rewards follow an elliptical distribution.

%\subsubsection{DRO chance constrained MDP}
However, in many real-world applications of chance constrained MDP, the true probability distribution of rewards or the transition probability is unknown. The decision-maker may only know partial information of the true distribution.
%Generally, only an estimate or a set of possible distributions is available.
It might lead to sub-optimal decisions or even failures of the decision-making system if the estimated distribution biased the true distribution.
The distributionally robust optimization (DRO) approach %provides a way to 
handles such ambiguity by making decisions against to the worst-case distribution in a set of all possible distributions, namely ambiguity set, rather than a single estimated distribution. There are two 
major 
types of ambiguity sets: the moments-based ones and the distance-based ones. In moments-based DRO \cite{wiesemann2014distributionally,delage2010distributionally,rahimian2019distributionally}, the decision-maker knows some moments information  
%the moments 
of the random parameters. In distance-based DRO \cite{hu2013kullback,gao2022distributionally,xie2021distributionally}, the decision-maker %knows
has a reference distribution and consider a ball centered at it with respect to a probability metric (say, Kullback-Leibler divergence, Wasserstein distance and so on),
given that she/he %believes
believes that the true distribution of random parameters is close to the reference distribution. Applying the DRO approach to chance constrained MDP, we have the distributionally robust chance constrained MDP (DRCCMDP) problem. As far as we know, only Nguyen et al. \cite{nguyen2022distributionally} studied the individual case of DRCCMDP 
%, they consider the case of 
with moments-based, $\phi$-divergence based and Wasserstein distance based ambiguity sets, respectively. However, the joint case of DRCCMDP has not been studied, which has important applications to guarantee the satisfaction of a complex system with multiple safety constraints in terms of a joint chance constraint.

%\subsubsubsection{DRO}
%literature review on DRO in other area. (2sentences. 6-7 references)

%1, DRO in expected value, mean-variance, CVaR [??]

%2, DRO ambiguity set moments [??], K-L[??], Wassertain [??]

%\subsubsubsection{DRO-MDP}
%literature review on DRO in MDP. (3-4 sentences. 4-5 references)

%\subsubsubsection{DRO-CC-MDP}
%literature review on DRO in chance constrained MDP. (1-2 sentences. no research in this topic, right ?)

%\subsection{Literature review on solution approaches}

%\subsubsection{optimization based}

%main: ---> optimization problem.
%SCOP, LP, 

%Abdel(CC-MDP)-> SOCP;

%Abdel(DRO-MDP)-> SOCP; --liang hong

%Abdel(DRO-CC-MDP)-> SOCP;

%However, optimization method shortfall: 1, 2, 3,.

%On the computational side to find the optimal policy in MDPs, the most adopted methods are by means of well studied optimization algorithms, 
Different optimization techniques have been widely used to find the optimal policy of an MDP problem, such as interior-point methods \cite{wright1997primal}, simplex method \cite{nelder1965simplex}. In practice, many solvers like GUROBI, MOSEK, CVX can be used to solve the resulting optimization problem directly. For instance, Delage and Mannor \cite{delage2010percentile} transformed the chance constrained MDP into a second-order cone programming (SOCP) and used a gradient descent algorithm to solve the SOCP to find the optimal policy. Varagapriya et al. \cite{varagapriya2022constrained} reformulated three kinds of constrained MDP problems into the linear programming (LP), SOCP and semi-definite programming (SDP), respectively, and solve them by CVX. Nguyen et al. \cite{nguyen2022distributionally} transformed the individual DRCCMDP into an SOCP or a mixed-integer SOCP when the ambiguity sets are moment-based, K-L divergence based or Wasserstein distance based, respectively. However, as for the joint DRCCMDP, it cannot directly be solved by an optimization solver due to the non-convexity of the joint chance constraint. And when the transformed optimization problem is large scale, the common optimization algorithms cannot find an accurate enough solution in a reasonable time. Therefore we choose dynamical neural network (DNN) as our solving approach to the joint DRCCMDP.
%However, the research on the reformulation of joint DRCCMDP and its corresponding solution method have not been studied.

%\subsubsection{AI-based}

%1, AI solve LP: Hopfield Tank 85; Wu dawen.

%2, AI qudratic, Xia, Nazemi 06

%3, AI SOCP, Ko et al.

%4, AI chance constraints, Nazemi, Abdel.

In this paper, we propose a dynamical neural network (DNN) approach to solve the joint DRCCMDP. The DNN approach is a typical  {AI based} technique that has been employed to solve optimization problems. Using DNNs to solve optimization problems was initiated by Hopfield and Tank \cite{hopfield1985neural}. Since then, the DNN approach has been extensively applied to solve different optimization problems, such as linear programming (LP) \cite{wang1993analysis,xia1996new}, second-order cone programming (SOCP) \cite{ko2011recurrent,nazemi2020new}, quadratic programming \cite{xia1996new1,nazemi2014neural,feizi2021solving}, nonlinear programming \cite{forti2004generalized,wu1996high}, minimax problems \cite{nazemi2011dynamical,gao2004neural}, stochastic game problems \cite{wu2022dynamical}, and geometric programming problems \cite{tassouli2023neural}, etc.
All of the work mentioned above adopts an ODE system to model the optimization problem, and then the optimal solution is obtained by solving the ODE system. The ODE systems have been shown to have global convergence properties, meaning that the state solutions converge to the optimal solution of the original optimization problem as the independent variable comes to infinity. In fact, Dissanayake et al. \cite{dissanayake1994neural} were the first to use a neural network to approximate the solution of differential equations, where the loss function contains two terms that satisfy the initial/boundary condition and the differential equation, respectively. Lagaris et al. \cite{lagaris1998artificial} developed a trial solution method that ensures initial conditions are always satisfied. Flamant et al. \cite{flamant2020solving} treat the parameters of ODE system models as an input variable to the neural network, allowing a neural network to be the solution of multiple differential equations, namely solution bundles. Neural networks nowadays have a large number of applications in computer vision, natural language processing, pattern recognition, and other fields \cite{dong2023adversarial,amiri2023novel}. It is certain that with the help of studies of ODE, the performance of DNNs, which transform the optimization problem into an ODE, can be continuously improved in the future. To the best of our knowledge, there is no study in the literature that uses a DNN approach to solve CCMDP problems.

%\subsubsection{observation}
%1, DRO-CC-MDP is 
% important bu not well studied.
%2, DRO-CC-MDP is non-convex, difficult. Trational opitmization based ?? local optimality fail to reach a good global optimal solution. However, high safe requirement ask for a high proba.. gurantee.  
%3, We try to apply the AI approach to DRO-CC-MDP. 

%\subsection{Contributions of our work}
After deriving a non-convex reformulation of the moment-based joint DRCCMDP problem, we apply the DNN approach to solve the obtained reformulation. As a comparison, we also apply the sequential convex approximation (SCA) algorithm \cite{hong2011sequential,liu2016stochastic} to solve the joint DRCCMDP problem, which decomposes the non-convex reformulation problem into a series of convex sub-problems. The main contributions of this paper can be summarized as follows:

% 1, our motivation
% 2, content
% 3, contribution

\begin{itemize}
    \item For the first time, we study the joint DRCCMDP under a moments-based uncertainty set. 
    %    derive the reformulation of the moment based  problem. And we prove its  biconvexity.
    \item To the best of our knowledge, this is the first attempt to apply the %For the first time, We first apply the 
    dynamic neural network approach to solve the CCMDP problem. 
    %We transform the oral optimization problem into its equivalent ordinary differential equation (ODE) system, then we solve it in a neural network way.
    \item %In the 
    Numerical experiments validate %, we analyze 
    the pros and cons of the dynamical neural network approach compared to the traditional sequential convex approximation approach.
    %in solving DRCCMDP. 
\end{itemize}

The structure of this paper is as follows. In Section \ref{model}, we introduce the DRCCMDP model. In Section \ref{joints}, we study the reformulation of the moments-based joint DRCCMDP. In Section \ref{DNN_J}, we apply the dynamical neural network approach to solve the moments-based joint DRCCMDP problem. We transform the reformulated optimization problem into an equivalent ordinary differential equation (ODE) system, and then solve it  {using a neural network way}. In Section \ref{numerical}, we apply the proposed DNN approach to a machine replacement problem and compare it with the SCA algorithm. In the last section, we conclude the paper.

  %  Section....   derive the reformulation of the moment based  problem. And we prove its  biconvexity.
   %We transform the oral optimization problem into its equivalent ordinary differential equation (ODE) system, then we solve it in a neural network way.

\section{Distributionally robust chance constrained MDP %DRCCMDP
}\label{model}
\subsection{MDP}
We consider an infinite horizon Markov decision process (MDP) as a tuple $(\mathcal{S}, \mathcal{A}, P, r_0, q, \alpha),$ where:
\begin{itemize}
\item[$\bullet$]$\mathcal{S}$ is a finite state space with $|S|$ states whose generic element is denoted by $s$ .
\item[$\bullet$]$\mathcal{A}$ is a finite action space with $|\mathcal{A}|$ actions and $a\in\mathcal{A}(s)$ denotes an action belonging to the set of actions at state $s.$
\item[$\bullet$]$P\in\mathbb{R}^{|\mathcal{S}|\times|\mathcal{A}|\times|\mathcal{S}|}$ is the distribution of transition probability $p(\overline{s}|s,a),$ which denotes the probability of moving from state $s$ to $\overline{s}$ when the action $a\in\mathcal{A}(s)$ is taken.
\item[$\bullet$]$r_0(s,a)_{s\in\mathcal{S}, a\in\mathcal{A}(s)}:\mathcal{S}\times\mathcal{A}\rightarrow\mathbb{R}$ denotes a running reward, which is the reward at the state $s$ when the action $a$ is taken. $r_0=(r_0(s,a))_{s\in\mathcal{S}, a\in\mathcal{A}(s)}\in\mathbb{R}^{|\mathcal{S}|\times|\mathcal{A}|}$ is the running reward vector.
\item[$\bullet$]$q=(q(s_0))_{s_0\in\mathcal{S}}$ is the probability of the initial state $s_0$.
\item[$\bullet$]$\alpha$ is the discount factor which satisfies $\alpha\in[0,1).$
\end{itemize}
In an MDP, the agent aims at maximizing her/his cumulative value with respect to the whole trajectory by choosing an optimal policy. We know from \cite{sutton1999policy} that there are two ways of formulating the agent's objective function. One is the average reward formulation and the other is considering a discount factor $\alpha\in[0,1).$ 
%As we care more about the long-term reward obtained from the MDP, we pay more attention on optimizing 
We assume that the agent cares more about current rewards than future rewards, and thus consider the discounting value function in this paper.

We consider a discrete time controlled Markov chain $(s_t,a_t)_{t=0}^{\infty}$ defined on the state space $\mathcal{S}$ and the action space $\mathcal{A}$, where $s_t$ and $a_t$ are the state and action at time $t$, respectively. At the initial time $t=0,$ the initial state is $s_{0}\in\mathcal{S},$ and the action $a_0\in\mathcal{A}(s_0)$ is taken according to the initial state's probability $q.$ Then the agent gains a reward $r_0(s_0,a_0)$ based on the current state and action. When $t=1,$ the state moves to $s_1$ with the transition probability $p(s_1|s_0,a_0).$ The dynamics of the MDP repeat at state $s_1$ and continue in the following infinite time horizon. 
%As a result, we are able to get the value function for the whole process.

We assume that running rewards $r_0$ and transition probabilities $p$ are stationary, i.e., they both only depend on states and actions rather than on time. We define the policy $\pi=(\mu(a|s))_{s\in\mathcal{S},a\in\mathcal{A}(s)}\in\mathbb{R}^{|\mathcal{S}|\times|\mathcal{A}|}$ where $\mu(a|s)$ denotes the probability that the action $a$ is taken at state $s$, and $\xi_t=\{s_0,a_0,s_1,a_1,...,s_{t-1}$, $a_{t-1},s_t\}$ the whole historical trajectory by time $t.$ 
%Let $\Theta_t$ be the set of all possible trajectories of length $t.$ For different time $t,$ 
Sometimes the decisions made by the agent may vary at different time $t$ accordingly, thus the policy may vary depending on time. We call this kind of policy the history dependent policy, denoted as $\pi_h=(\mu_{t}(a|s))_{s\in\mathcal{S},a\in\mathcal{A}(s)}, t=1,2,...,\infty.$ When the policy is independent of time, we call it a stationary policy. That is, there exists a vector $\overline{\pi}$ such that $\pi_h=(\mu_{t}(a|s))_{s\in\mathcal{S},a\in\mathcal{A}(s)}=\overline{\pi}=(\overline{\mu}(a|s))_{s\in\mathcal{S},a\in\mathcal{A}(s)}$ for all $t.$ Let $\Pi_h$ and $\Pi_s$ be the sets of all possible history dependent policies and stationary policies, respectively. When the reward $r_0(s,a)$ is random, for a fixed $\pi_h\in\Pi_h,$ we consider the discounted expected value function 
\begin{equation}\label{00}
    V_{\alpha}(q,\pi_h)=\sum_{t=0}^{\infty}\alpha^{t}\mathbb{E}_{q,\pi_h}(r_0(s_t,a_t)),
\end{equation}
where $\alpha\in [0,1)$ is the given discount factor. The object of the agent is to maximize the discounted expected value function
\begin{equation}\label{optim MDP}
    \max_{\pi_{h}\in\Pi_h}{\sum_{t=0}^{\infty}\alpha^{t}\mathbb{E}_{q,\pi_h}(r_0(s_t,a_t))}.
\end{equation}

We denote by $d_{\alpha}(q,\pi_h,s,a)$ the $\alpha$-discounted occupation measure \cite{altman1999constrained} such that 
$$
\begin{aligned}
d_{\alpha}(q,\pi_h,s,a) 
%& = (1-\alpha)\sum_{t=0}^{\infty}\alpha^{t}q(s)p(s|s,a)\mu_t(a|s) \\
=(1-\alpha)\sum_{t=0}^{\infty}\alpha^{t}p_{q,\pi_{h}}(s_{t}=s, a_{t}=a),\ \  \forall{s\in\mathcal{S},\ \  a\in\mathcal{A}(s)}.
\end{aligned}
$$
%\begin{subequations}
%\begin{eqnarray}
%& d_{\beta}(q,\pi_h,s,a) & =(1-\beta)\sum_{t=0}^{\infty}\beta^{t}q(s)p(s|s,a)\mu_t(a|s) \\
%&& =(1-\beta)\sum_{t=0}^{\infty}\beta^{t}p_{q,\pi_{h}}(s_{t}=s, a_{t}=a), \forall{s\in\mathcal{S},a\in\mathcal{A}(s)}.
%\end{eqnarray}
%\end{subequations} 
As the state and action spaces are both finite, by Theorem 3.1 in \cite{altman1999constrained}, the occupation measure $d_{\alpha}(q,\pi_h,s,a)$ is a well-defined probability distribution. By taking the occupation measure into \eqref{00}, the discounted expected value function can be written as
$$
\begin{aligned}
V_{\alpha}(q,\pi_h) & =\sum_{(s,a)\in\Lambda}\sum_{t=0}^{\infty}\alpha^{t}p_{q,\pi_h}(s_{t}=s,a_{t}=a) r_{0}(s,a) \\
& =\frac{1}{1-\alpha}\sum_{(s,a)\in\Lambda\label{value}}d_{\alpha}(q,\pi_h,s,a) r_0(s,a),
\end{aligned} 
$$ where $\Lambda=\left\{(s,a)|s\in\mathcal{S}, a\in\mathcal{A}(s)\right\}$.

It is known from Theorem 3.2 in \cite{altman1999constrained} that the set of occupation measures corresponding to history dependent policies is equal to that corresponding to stationary ones. We thus have: 
% \begin{lemma}[\cite{varagapriya2022constrained}]\label{th1}
%     The set of occupation measures corresponding to history dependent policies is equal to the set
% \begin{equation}\label{le}
%     \Delta_{\alpha,q}=\left\{\tau \in \mathbb{R}^{|\mathcal{S}|\times|\mathcal{A}|} \mid \begin{array}{l}
% \sum\limits_{(s,a)\in\Lambda}\tau(s,a)\left(\delta(s',s)-\alpha p(s'|s,a)\right)=(1-\alpha)q(s'), \\
% \tau(s,a)\ge0, \forall s',s\in\mathcal{S},a\in\mathcal{A}(s).
% \end{array}\right\},
% \end{equation} where $\delta(s',s)$ is the Kronecker delta, such that the expected discounted value function defined by \eqref{value} remains the same. 
% \end{lemma}
\begin{lemma}[\cite{varagapriya2022constrained,altman1999constrained}]\label{th1}
    The set of occupation measures corresponding to history dependent policies is equal to the set
\begin{equation}\label{le}
    \Delta_{\alpha,q}=\left\{\tau \in \mathbb{R}^{|\mathcal{S}|\times|\mathcal{A}|}\Bigg|
    \begin{array}{l}
\sum\limits_{(s,a)\in\Lambda}\tau(s,a)\left(\delta(s',s)-\alpha p(s'|s,a)\right)=(1-\alpha)q(s'), \\
\tau(s,a)\ge0, \forall s',s\in\mathcal{S},a\in\mathcal{A}(s).
\end{array}\right\},
\end{equation} where $\delta(s',s)$ is the Kronecker delta, such that the expected discounted value function defined by \eqref{optim MDP} remains invariant to time $t$. 
\end{lemma}

With Lemma \ref{th1}, the MDP problem \eqref{optim MDP} with history dependent policies is equivalent to the following optimization problem:
\begin{subequations}\label{9pp}
\begin{eqnarray}
& \max\limits_{\tau} & \frac{1}{1-\alpha}\sum\limits_{(s,a)\in\Lambda}\tau(s,a)r_0(s,a)\\
&{\rm s.t.} & \tau\in\Delta_{\alpha,q}.
\end{eqnarray}
\end{subequations}

\begin{remark}\label{TM}
If $\tau^{\star}$ is an optimal solution of \eqref{9pp}, then the optimal stationary policy of \eqref{optim MDP} is $\frac{\tau^{\star}(s,a)}{\sum\limits_{a\in\mathcal{A}(s)}\tau^{\star}(s,a)}$ for all $(s,a)\in \Lambda$, whenever the denominator is nonzero.
\end{remark}

%\subsection{Constrained MDP}
In a constrained MDP (CMDP), on the basis of the MDP defined above, we consider both running objective rewards in the objective function (denoted by $r_0$ introduced pre-ahead) and some running constraint rewards in some safe constraints. Let $r_{k}(s,a)_{(s,a)\in\Lambda}:\mathcal{S}\times\mathcal{A}\rightarrow\mathbb{R}, k=1,2,...,K,$ be the running constraint rewards and $r_k=(r_{k}(s,a))_{(s,a)\in\Lambda}\in\mathbb{R}^{|\Lambda|}$ be the running constraint rewards vector. Let $\Xi=(\xi_k)_{k=1}^{K}$ be the bounds for the constraints. A CMDP is defined by the tuple $(\mathcal{S}, \mathcal{A}, R, \Xi, P, q, \alpha)$ where $R=(r_k)_{k=0}^{K}.$

Considering the natural situation that more attention is paid on optimizing current rewards than future ones, we apply the discount factor in the expected constrained value function and define 
\begin{equation}
    \phi_{k,\alpha}(q,\pi_h)=\frac{1}{1-\alpha}\sum_{(s,a)\in\Lambda}d_{\alpha}(q,\pi_h,s,a) r_k(s,a)
\end{equation} to be the $k$-th constrained expected value function. Combined with Lemma \ref{th1}, the object of the agent in a CMDP can be written as the following optimization problem
\begin{subequations}
\begin{eqnarray}
& \max\limits_{\tau} & \frac{1}{1-\alpha}\sum\limits_{(s,a)\in\Lambda}\tau(s,a)r_0(s,a)\\
&{\rm s.t.} & \sum\limits_{(s,a)\in\Lambda}\tau(s,a)r_k(s,a)\ge \xi_{k}, \ \  k=1,2,...,K,\\
&&\tau\in\Delta_{\alpha,q}.
\end{eqnarray}
\end{subequations}

% \subsection{RCMDP}
% Based on the definition of MDP above, we assume that the rewards vectors $r_k, k=0,1,...,K$ are random, and the transition probabilities are known. A most common and useful approach to handle the uncertainty is robust optimization. That is we assume the uncertain parameters are constrained to be in a set, which is uncertain. And we consider the worst-case scenario over the set to solve the original optimization problem.

% When applying the approach of robustness, we get the robust constrained MDP (RCMDP), which can be defined by the tuple $(\mathcal{S}, \mathcal{A}, R, \Xi, P, \Omega_{R}, q, \alpha)$. In this tuple, $\Omega_{R}=\bm{\bigtimes}_{k=0}^{K}\Omega_{r_k}$ and $\Omega_{r_k}$ denotes the uncertain set of the running rewards. In a RCMDP with random rewards and deterministic transition probabilities, the agent aims at solving the following optimization problem under the worst-case,
% \begin{subequations}
% \begin{eqnarray}
% \max\limits_{\tau} & \inf\limits_{r_0\in\Omega_{r_0}} & \frac{1}{1-\alpha}\sum\limits_{(s,a)\in\Lambda}\tau(s,a)r_0(s,a)\\
% {\rm s.t.} & \inf\limits_{r_k \in\Omega_{r_k}} & \sum\limits_{(s,a)\in\Lambda}\tau(s,a)r_k(s,a)\ge \xi_{k},  k=1,2,...,K\\
% &&\tau\in\Delta_{\alpha,q}.
% \end{eqnarray}
% \end{subequations}

\subsection{Distributionally robust chance constrained MDP}
In many applications, the reward vectors $r_{k}, k=0,1,...,K$ are random. It's reasonable for us to consider the MDP with random rewards. We use the expected value of the objective rewards as the objective function.  {To ensure a high satisfaction probability of safe constraints, we use a chance constraint for the safe constraints in the CMDP.} We denote such an MDP problem as the chance-constrained MDP (CCMDP).

Let $F$ denote the joint probability distribution of $r_1, r_2,...,r_K$ and $\hat{\epsilon}$ denote the confidence level for the $K$ constraints. Then, the joint CCMDP (J-CCMDP) can be defined  as the following optimization problem
\begin{subequations}
\begin{eqnarray}
\rm{(J-CCMDP)} & \max\limits_{\tau} & \frac{1}{1-\alpha}\mathbb{E}_{F_0}[\tau^{\top}\cdot r_{0}]\\
&{\rm s.t.} & \mathbb{P}_{F}(\tau^{\top}\cdot r_k\ge \xi_{k},  k=1,2,...,K)\ge\hat{\epsilon}, \\
&&\tau\in\Delta_{\alpha,q}
\end{eqnarray}
\end{subequations} on a tuple $(\mathcal{S}, \mathcal{A}, R, \Xi, P, F, q, \alpha, \hat{\epsilon})$, where $\hat{\epsilon}\in[0,1]$.

%For the $k$-th random constrained rewards vector $r_k=(r_{k}(s,a))_{(s,a)\in\Lambda}$, we assume its probability distribution is $F_k$, $k=0,1,...,K$. We preset a confidence vector $\epsilon=(\epsilon_k)_{k=1}^{K}$ for the CCMDP, where $\epsilon_k\in[0,1]$. Then we can define the individual CCMDP (I-CCMDP) as a tuple $(\mathcal{S}, \mathcal{A}, R, \Xi, P, \mathcal{D}, q, \alpha, \epsilon)$, where $\mathcal{D}=(F_k)_{k=0}^{K}$, which can be reformulated as the following optimization problem:
When the satisfactions of different constraints are considered separately, we obtain the following individual CCMDP, % as a special case of the joint one,
\begin{subequations}
\begin{eqnarray}
\rm{(I-CCMDP)} & \max\limits_{\tau} & \frac{1}{1-\alpha}\mathbb{E}_{F_0}[\tau^{\top}\cdot r_0]\\
&{\rm s.t.} & \mathbb{P}_{F_k}(\tau^{\top}\cdot r_k\ge \xi_{k})\ge\epsilon_k,  k=1,2,...,K\\
&&\tau\in\Delta_{\alpha,q},
\end{eqnarray}
\end{subequations} where $F_{k}$ is the probability distribution of $r_{k}$ and $\epsilon_{k}\in[0,1]$ is the confidence level of the $k-$th constraint.

If the distribution information of rewards $r_k$ are not perfectly known, we can apply the distributionally robust optimization approach to handle the ambiguity of $\hat{F}$ or $F_{k},k=0,...,K$. Then we consider a distributionally robust chance constrained MDP (DRCCMDP). 
The joint DRCCMDP (J-DRCCMDP) can be defined as the tuple $(\mathcal{S}, \mathcal{A}, R, \Xi, P, \mathcal{D}, \mathcal{F}_0, \mathcal{F}, q, \alpha, \hat{\epsilon})$, where $\hat{\epsilon}\in[0,1]$, $\mathcal{F}_0$ is the ambiguity set for the unknown
%random 
distribution $F_0$ and $\mathcal{F}$ is the ambiguity set for the unknown %random
joint distribution $F$ of $r_{1}, r_2,...,r_k$. The resulting J-DRCCMDP can be formulated as
\begin{subequations}\label{Jobj}
\begin{eqnarray}
\rm{(J-DRCCMDP)}  & \max\limits_{\tau}&\inf\limits_{F_{0}\in\mathcal{F}_0} \  \frac{1}{1-\alpha}\mathbb{E}_{F_0}[\tau^{\top}\cdot r_0]\label{12a}\\
 & {\rm s.t.} & \inf\limits_{F\in\mathcal{F}}\  \label{JF}\mathbb{P}_{F}(\tau^{\top}\cdot r_k\ge \xi_{k}, k=1,2,...,K)\ge\hat{\epsilon}, \\
&&\tau\in\Delta_{\alpha,q}.
\end{eqnarray}
\end{subequations}

As a speical case, the individual DRCCMDP (I-DRCCMDP) 
%can be defined as the tuple $(\mathcal{S}, \mathcal{A}, R, \Xi, P, \mathcal{D}, \tilde{\mathcal{F}}, q, \alpha, \epsilon)$, where $\tilde{\mathcal{F}}=(\mathcal{F}_k)_{k=0}^{K}$ when . Therefore the I-DRCCMDP 
can be written as the following optimization problem:
\begin{subequations}\label{obj MDP}
\begin{eqnarray}
\rm{(I-DRCCMDP)}  & \max\limits_{\tau} & \inf\limits_{F_{0}\in\mathcal{F}_0} \   \frac{1}{1-\alpha}\mathbb{E}_{F_0}[\tau^{\top}\cdot r_0] \label{kj} \\
& {\rm s.t.} & \inf\limits_{F_{k}\in\mathcal{F}_k} \ \mathbb{P}_{F_k}\label{pd}(\tau^{\top}\cdot r_k\ge \xi_{k})\ge\epsilon_k,\   k=1,2,...,K, \\
&&\tau\in\Delta_{\alpha,q},
\end{eqnarray}
\end{subequations} where $\mathcal{F}_k$ is the ambiguity set for the distribution $F_k$ and $\epsilon_{k}\in[0,1]$ is the confidence level of the $k-$th constraint.

\section{moment based J-DRCCMDP}\label{joints}

%\subsection{moment based J-DRCCMDP}
In this section, we consider the reformulation of the J-DRCCMDP problem $\eqref{Jobj}$ when we only know information about the first two order moments of the joint distribution.

\subsection{Reformulation of moment based J-DRCCMDP}
We first consider the ambiguity of the marginal distribution.
Applying the Mahalanobis distance, we define the ambiguity of the mean vector through an ellipsoid of size centered at $\mu_k$ for each $k=0,1,...,K$. And we specify the covariance matrix of the reward vector lies in a positive semi-definite cone, which gives a one-side estimation on the covariance. In detail, we define
\begin{equation}\label{moments1}
    \mathcal{F}_k=\left\{F_k \Bigg| \begin{array}{l}
(\mathbb{E}_{F_k}[r_k]-\mu_k)^{\top}(\Sigma_k)^{-1}(\mathbb{E}_{F_k}[r_k]-\mu_k)\le\rho_{1,k}, \\
{\rm{Cov}}_{F_k}[r_k]\preceq_{S}\rho_{2,k}\Sigma_k.
\end{array}\right\},
\end{equation} where ${\rm{Cov}}_{F_k}[r_k]=\mathbb{E}_{F_k}\left[(r_{k}-\mathbb{E}_{F_k}[r_k])^{\top}(r_{k}-\mathbb{E}_{F_k}[r_k])\right]$, $\mu_{k}=[\mu_{0}^k, \mu_{1}^k,..., \mu_{|\Lambda|}^k]^{\top}$ and $\Sigma_k=\{\sigma_{i,j}^{k},i,j=1,2,...,|\Lambda|\}$ for $k=0,1,2,...,K$. Here, $\mu_{i}^{k}$ is the reference value of the expected value of $r_{k}(i)$, $i=1,2,...,|\Lambda|, k=0,1,...,K$ and $\sigma_{i,j}^k$ is the reference value of the covariance between $r_{k}(i)$ and $r_{k}(j)$, $i,j=1,2,...,|\Lambda|, k=0,1,...,K$. We assume that $\Sigma_k$ is a positive semidefinite matrix, $\rho_{1,k},\rho_{2,k}$ are two nonnegative parameters controlling the size of the uncertainty sets. $A\preceq_{S}B$ means $B-A\in\mathbb{S}^n$, where $\mathbb{S}^{n}$ is the $n$-dimensional positive semidefinite cone.

We assume that different rows in the joint chance constraint are independent of each other and consider the following ambiguity set for the joint distribution 
\begin{equation}\label{Join}
\mathcal{F}:=\mathcal{F}_{1}\times\cdots\times\mathcal{F}_{K}=\left\{F=F_{1}\times\cdots\times F_{K}|F_{k}\in\mathcal{F}_{k}, k=1,...,K \right\},
\end{equation} where $F$ is the joint distribution for random vectors $r_1,...,r_K$ with marginals $F_1,...,F_K$ and $\mathcal{F}_{1},...,\mathcal{F}_{K}$ are all defined in \eqref{moments1}. Then we have the following proposition.
\begin{proposition}
%When we know information about the first two moments of the rewards' joint distribution $F$ as defined above, the objective optimization problem \eqref{Jobj} has the following reformulation when different rows are independent of each other:
Given the ambiguity set $\mathcal{F}$ defined in \eqref{Join}, the J-DRCCMDP problem \eqref{Jobj} can be reformulated as:
\begin{subequations}\label{gh}
\begin{eqnarray}
& \min\limits_{\tau\in\mathbb{R}_{+}^{|\Lambda|}, h\in\mathbb{R}_{+}^{K}} & \frac{1}{1-\alpha}\left[-\tau^{\top}\mu_{0}+\sqrt{\rho_{1,0}}\Vert{(\Sigma_0)^{\frac{1}{2}}\tau}\Vert\right]\\
&{\rm s.t.} & \tau^{\top}\mu_{k}-\left(\sqrt{\frac{h_k}{1-\label{2r}h_k}}\sqrt{\rho_{2,k}}+\sqrt{\rho_{1,k}}\right)\Vert{(\Sigma_k)^{\frac{1}{2}}\tau}\Vert\ge\xi_{k},  k=1,2,...,K,\\
&&0\le h_{k}\le \label{3r}1,k=1,2,...,K,\\
&&\prod_{k=1}^{K}h_{k}\ge\hat{\epsilon},\\
&&\tau\in\Delta_{\alpha,q}.\label{4r}
\end{eqnarray}
\end{subequations}
\end{proposition}
\begin{proof}
%The reformulation of objective function $\eqref{1r}$ is the same as that in Proposition $\ref{the1}$.
Given the ambiguity set $\mathcal{F}_{0}$ as defined in \eqref{moments1}, we find that \eqref{12a} has the same structure as (7) in \cite{liu2022distributionally}, so by Theorem 1 in \cite{liu2022distributionally}, \eqref{12a} has the following equivalence 
\begin{equation}\label{9845}
    \min\limits_{\tau\in\mathbb{R}_{+}^{|\Lambda|}} \frac{1}{1-\alpha}\left[-\tau^{\top}\mu_{0}+\sqrt{\rho_{1,0}}\Vert{(\Sigma_0)^{\frac{1}{2}}\tau}\Vert\right].
\end{equation} Given the ambiguity set $\mathcal{F}$ defined in \eqref{Join}, a joint DRO chance constraint is equivalent to the product of DRO counterparts of individual marginal distributions \cite{shapiro2021lectures}. This means that
$$\eqref{JF}\iff\prod_{k=1}^{K}\inf\limits_{F_{k}\in\mathcal{F}_k}\mathbb{P}_{F_k}(\tau^{\top}r_k\ge \xi_{k})\ge\hat{\epsilon}.$$ By introducing the auxiliary variables $h_{k}>0, k=1,2,...,K$, \eqref{JF} is equivalent to the following reformulation:
\begin{subequations}
\begin{eqnarray}
&\inf\limits_{F_{k}\in\mathcal{F}_k}\mathbb{P}_{F_k}(\tau^{\top}\label{iva} r_k\ge \xi_{k})\ge h_k,k=1,2,...,K,\label{16aa}\\
&\prod\limits_{k=1}\limits^{K}h_{k}\ge\hat{\epsilon},\ \  0\le h_{k}\le 1,\ \  k=1,2,...,K.\label{16bb}     
\end{eqnarray} 
\end{subequations} We can see that \eqref{iva} exhibits the same structure as the individual distributionally robust chance constrained geometric program with moments-based ambiguity set in \cite{liu2022distributionally}. Then, by Theorem 1 in \cite{liu2022distributionally}, \eqref{iva} has the following equivalent reformulation: 
\begin{equation}\label{irgp}
\tau^{\top}\mu_{k}-\left(\sqrt{\frac{h_k}{1-h_k}}\sqrt{\rho_{2,k}}+\sqrt{\rho_{1,k}}\right)\Vert{(\Sigma_k)^{\frac{1}{2}}\tau}\Vert\ge\xi_{k},  k=1,2,...,K.
\end{equation}
Combining with the reformulations \eqref{9845}, \eqref{16aa}, \eqref{16bb} and \eqref{irgp} gives us \eqref{gh}.
\end{proof}

\begin{remark}
%Furthermore, if $\tau>0$ then the reformulation in Proposition \eqref{the1} can be transformed with the logarithmic technique into the next reformulation by taking the vector $-\tau$ into vector $-y$, where $y(i)=\ln{\tau(i)}, i=1,2,...,\mathbb{K}.$:
By \cite{liu2016stochastic,liu2022distributionally,bartlett1946statistical}, we know that the logarithmic transformation is a common method to convert a nonconvex optimization problem into a convex one. After introducing $\tilde{\tau}(i)=\ln{\tau(i)}, i=1,2,...,|\Lambda|$ with the assumption that $\tau>0$, we have a reformulation of \eqref{gh}
\begin{subequations}\label{eq2}
\begin{eqnarray}
& \min\limits_{\tilde{\tau}\in\mathbb{R}^{|\Lambda|},h\in\mathbb{R}_{+}^{K}} & -\mu_{0}^{\top}e^{\tilde{\tau}-\log{(1-\alpha)}\cdot1_{|\Lambda|}}+\Vert(\Sigma_{0})^{\frac{1}{2}}e^{\tilde{\tau}+\left(\frac{1}{2}\log{(\rho_{1,0})}-\log{(1-\alpha)}\right)1_{|\Lambda|}}\Vert\label{17a}\\
&{\rm s.t.} & \mu_{k}^{\top}e^{\tilde{\tau}}-\Vert(\Sigma_{k})^{\frac{1}{2}}e^{\tilde{\tau}+\log{(\sqrt{\frac{h_k}{1-h_k}}\sqrt{\rho_{2,k}}+\sqrt{\rho_{1,k}})}\cdot1_{|\Lambda|}}\Vert\ge\xi_k,  k=1,2,...,K\\\label{17c}
&&0\le h_{k}\le 1,k=1,2,...,K,\\    
&&\prod_{k=1}^{K}h_{k}\ge\hat{\epsilon},\\ &&\tilde{\tau}\in\tilde{\Delta}_{\alpha,q},\label{17e}
\end{eqnarray}
\end{subequations} where 
\begin{equation}\label{del}
    \tilde{\Delta}_{\alpha,q}=\left\{\tilde{\tau} \in \mathbb{R}^{|\Lambda|} \mid \begin{array}{l}
\sum\limits_{(s,a)\in\Lambda}e^{\tilde{\tau}(s,a)}\left(\delta(s',s)-\alpha p(s'|s,a)\right)=(1-\alpha)q(s'), \forall s',s\in\mathcal{S},a\in\mathcal{A}(s)
\end{array}\right\}. 
\end{equation} It is worth mentioning that the optimization problem \eqref{eq2} is still nonconvex due to \eqref{17e}, though \eqref{17a}-\eqref{17c} have a convex form. This can be illustrated as follows.

For any $\tilde{\tau}_{1}, \tilde{\tau}_{2}\in\tilde{\Delta}_{\alpha,q}$, and $\forall\lambda\in[0,1],s'\in\mathcal{S}$, we have
\begin{subequations}
\begin{eqnarray}
&&\sum\limits_{(s,a)\in\Lambda}e^{\lambda \tilde{\tau}_{1}+(1-\lambda)\tilde{\tau}_{2}}\left(\delta(s',s)-\alpha p(s'|s,a)\right)\\
&=&\sum\limits_{(s,a)\in\Lambda}e^{\tilde{\tau}_1}e^{(\lambda-1)(\tilde{\tau}_{1}-\tilde{\tau}_{2})}\left(\delta(s',s)-\alpha p(s'|s,a)\right)\\
&=&e^{(\lambda-1)(\tilde{\tau}_{1}-\tilde{\tau}_{2})}\sum\limits_{(s,a)\in\Lambda}e^{\tilde{\tau}_1}\left(\delta(s',s)-\alpha p(s'|s,a)\right)\\
&=&e^{(\lambda-1)(\tilde{\tau}_{1}-\tilde{\tau}_{2})}(1-\alpha)q(s')
\end{eqnarray} 
\end{subequations} There exists some $\tilde{\tau}_{1}, \tilde{\tau}_{2}\in\tilde{\Delta}_{\alpha,q}$ and $\lambda\in[0,1]$ such that $e^{(\lambda-1)(\tilde{\tau}_{1}-\tilde{\tau}_{2})}\neq1$, thus the set $\tilde{\Delta}_{\alpha,q}$ is nonconvex and the optimization problem \eqref{eq2} is a nonconvex one.

\end{remark}

\subsection{Sequential convex approximation algorithm for J-DRCCMDP}

%Next, we propose our way to transform the optimization problem \eqref{gh} into a convex one for its nonconvexity.
Like \cite{hong2011sequential} and \cite{liu2016stochastic}, we can use the sequential convex approximation approach to deal with the nonconvexity of \eqref{gh}. The basic idea of the sequential convex approximation consists in decomposing the original problem into subproblems where a subset of variables is fixed alternatively. For problem \eqref{gh}, we first fix $h=h^{n}$ and update $\tau$ by solving
\begin{subequations}\label{77}
\begin{eqnarray}
& \min\limits_{\tau\in\mathbb{R}_{+}^{|\Lambda|}} & \frac{1}{1-\label{1r}\alpha}\left[-\tau^{\top}\mu_{0}+\sqrt{\rho_{1,0}}\Vert{(\Sigma_0)^{\frac{1}{2}}\tau}\Vert\right]\\
&{\rm s.t.} & \tau^{\top}\mu_{k}-\left(\sqrt{\frac{h_{k}^{n}}{1-h_{k}^{n}}}\sqrt{\rho_{2,k}}+\sqrt{\rho_{1,k}}\right)\Vert{(\Sigma_k)^{\frac{1}{2}}\tau}\Vert\ge\xi_{k},  k=1,2,...,K,\\
&&\tau\in\Delta_{\alpha,q},
\end{eqnarray}
\end{subequations} and then we fix $\tau=\tau^{n}$ and update $h$ by solving
\begin{subequations}\label{88}
\begin{eqnarray}
& \min\limits_{h\in\mathbb{R}^{K}_{+}} & \sum_{k=1}^{K}\psi_{k}h_{k}\\
&{\rm s.t.} & h_{k}\le\frac{\mathscr{A}_{k}^2}{1+\mathscr{A}_{k}^2}, k=1,2,...,K,\\
&&  0\le h_{k}\le 1, k=1,2,...,K,\label{21c}\\
&& \sum_{k=1}^{K}\log{h_{k}}\ge\log{\hat{\epsilon}},\label{21d}
\end{eqnarray}
\end{subequations} where $\mathscr{A}_{k}=\frac{{\tau}^{\top}\mu_{k}-\xi_{k}}{\Vert(\Sigma_k)^{\frac{1}{2}}\tau\Vert\sqrt{\rho_{2,k}}}-\sqrt{\frac{\rho_{1,k}}{\rho_{2,k}}}$ and $\psi_k$ is a given searching direction for $h_k$, $k=1,\ldots,K$. The detailed procedure of the sequential convex approximation algorithm for solving problem \eqref{gh} is given in Algorithm 1.

\begin{algorithm}[htp]
\SetAlgoLined
 \KwData{$\mu_{k}$, $\Sigma_{k}$, $\rho_{1,k}$, $\rho_{2,k}$, $\xi_{k}$, $\Delta_{\alpha,q}$, $n_{max}$, $\gamma$, $\hat{\epsilon}$, $L$, $k=0,1,...,K$.}
  \KwResult{$\tau^{n}$, $V^{n}$.}
Set $n=0$\;
Choose an initial point $h^{0}$ which is feasible for \eqref{21c}-\eqref{21d}\;
\While {$n\le n_{max}$ and $\Vert h^{n-1}-h^{n}\Vert\ge L$}{
Solve problem \eqref{77}; let $\tau^{n}, \theta^{n}, V^{n}$ be an optimal solution, the optimal Lagrangian dual variable and the optimal value of \eqref{77}, respectively\;
Solve problem \eqref{88} with
$$\mathscr{A}_{k}=\frac{{\tau^{n}}^{\top}\mu_{k}-\xi_{k}}{\Vert(\Sigma_k)^{\frac{1}{2}}\tau^{n}\Vert\sqrt{\rho_{2,k}}}-\sqrt{\frac{\rho_{1,k}}{\rho_{2,k}}}, \ \ \psi_{k}=\theta^{n}_{k}\frac{\Vert(\Sigma_k)^{\frac{1}{2}}\tau^{n}\Vert}{2(1-h_{k}^{n})}\sqrt{\frac{\rho_{2,k}}{h_{k}^{n}(1-h_{k}^{n})}},$$ and let $\tilde{h}$ be an optimal solution of \eqref{88}\;
$h^{n+1}\leftarrow h^{n}+\gamma(\tilde{h}-h^{n}), n\leftarrow n+1$. Here, $\gamma\in (0,1)$ is the step size.
}
\caption{Sequential convex approximation algorithm(Problem \eqref{gh})}
\end{algorithm} 

\

\par
 Algorithm 1 can be seen as a specific relization of the alternate convex search or block-relaxation methods \cite{gorski2007biconvex}. By Theorem 2 in \cite{liu2016stochastic}, we know that Algorithm 1 converges in a finite number of iterations and the returned value $V^{n}$ is an upper bound of problem \eqref{gh}. When these sub-problems are all convex, the objective function is continuous, and the feasible set is closed, the alternate convex search algorithm converges monotonically to a partial optimal point (Theorem 4.7 \cite{gorski2007biconvex}). When the objective function is a differentiable and biconvex function, $(h,\tau)$ is a partial optimal point if and only if $(h,\tau)$ is a stationary point (Corollary 4.3 \cite{gorski2007biconvex}). As $\Sigma_{k}$ is positive semi-definite, $\Vert(\Sigma_k)^{\frac{1}{2}}\tau^{n}\Vert$ is a convex function for any $k=1,2,...,K$. Thus both problems \eqref{77} and \eqref{88} are convex programs and Algorithm 1 converges to a stationary point.

\section{Dynamical neural network approach for J-DRCCMDP}\label{DNN_J}
Different from the current methodology, we consider in this section using the DNN approach to solve problem \eqref{gh}.
%We use a DNN to solve the KKT conditions problem as 
We show that the equilibrium point of the DNN model corresponds to the KKT point of problem \eqref{gh}. Then we study the stability of the equilibrium point by analyzing the Lyapunov function. Figure 1 shows how we apply the DNN approach to solve the J-DRCCMDP problem.
\begin{figure}[htb]
    \centering
    \includegraphics[width=3cm]{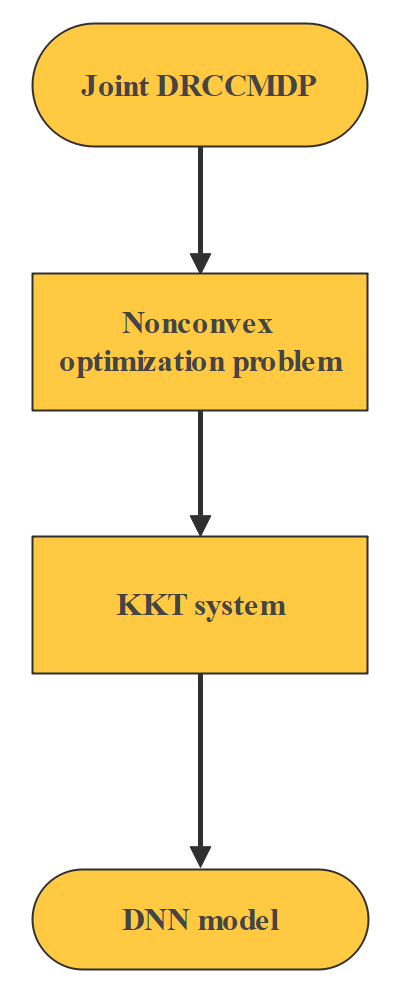}
    \caption{Flowchart of the DNN approach for solving J-DRCCMDP}
\end{figure}

Firstly, we apply the log-transformation $h_{k}=e^{x_k}, k=1,2,...,K$ to problem \eqref{gh} and we get the following reformulation.
\begin{subequations}\label{399}
\begin{eqnarray}
& \min\limits_{\tau\in\mathbb{R}_{+}^{|\Lambda|}, x\in\mathbb{R}_{-}^{K}} & \frac{1}{1-\alpha}\left[-\tau^{\top}\mu_{0}+\sqrt{\rho_{1,0}}\Vert{(\Sigma_0)^{\frac{1}{2}}\tau}\Vert\right]\\
&{\rm s.t.} & \tau^{\top}\mu_{k}-\left(\sqrt{\frac{e^{x_k}}{1-e^{x_k}}}\sqrt{\rho_{2,k}}+\sqrt{\rho_{1,k}}\right)\Vert{(\Sigma_k)^{\frac{1}{2}}\tau}\Vert\ge\xi_{k},  k=1,2,...,K\label{35a}\\
&&x_{k}\le 0,k=1,2,...,K,\\
&&\sum\limits_{k=1}^{K}x_{k}\ge\log\hat{\epsilon},\\ &&\tau\in\Delta_{\alpha,q}.
\end{eqnarray}
\end{subequations}
%We claim that \eqref{399} is a biconvex optimization problem.
\begin{lemma}
The optimization problem \eqref{399} is a biconvex optimization problem with respect to $\tau$ and $x$.
\end{lemma}
\begin{proof}
    For fixed $x\in \mathbb{R}_{-}^{K}$, \eqref{399} is a SOCP problem, thus it is convex. We just need to show the convexity of \eqref{399} with fixed $\tau\in \mathbb{R}_{+}^{|\Lambda|}$, which requires the convexity of \eqref{35a}. When $\tau$ is fixed, \eqref{35a} can be written as $\phi_k(x_k)\le 0$, where $$\phi_{k}(x_k)=\left(\sqrt{\frac{e^{x_k}}{1-e^{x_k}}}\sqrt{\rho_{2,k}}+\sqrt{\rho_{1,k}}\right)\Vert{(\Sigma_k)^{\frac{1}{2}}\tau}\Vert-\tau^{\top}\mu_{k}+\xi_{k}.$$ 
    %Next we prove that $\phi_k(x_k)$ is a convex function.

    We have that $$\phi_{k}^{'}(x_k)=\frac{e^{x_k}}{2\sqrt{\frac{e^{x_k}}{1-e^{x_k}}}(1-e^{x_k})^{2}}\sqrt{\rho_{2,k}}\Vert{(\Sigma_k)^{\frac{1}{2}}\tau}\Vert,$$ and $$\phi_{k}^{''}(x_k)=\frac{e^{x_k}(2e^{x_k}+1)}{4\sqrt{\frac{e^{x_k}}{1-e^{x_k}}}(1-e^{x_k})^3}\sqrt{\rho_{2,k}}\Vert{(\Sigma_k)^{\frac{1}{2}}\tau}\Vert.$$ As $x\in \mathbb{R}_{-}^{K}$, then $\phi_{k}^{''}(x_k)\ge 0$, $\forall x_{k}\le 0$. Therefore 
    %\eqref{35a} is convex when $\tau$ is fixed, and 
    \eqref{399} is a convex optimization problem when $\tau$ is fixed. Thus \eqref{399} is biconvex.
\end{proof}

\subsection{KKT conditions}
Let $f(\tau)=\frac{1}{1-\alpha}\left[-\tau^{\top}\mu_{0}+\sqrt{\rho_{1,0}}\Vert{(\Sigma_0)^{\frac{1}{2}}\tau}\Vert\right]$, $\phi_{k}(\tau, x)=\left(\sqrt{\frac{e^{x_k}}{1-e^{x_k}}}\sqrt{\rho_{2,k}}+\sqrt{\rho_{1,k}}\right)\Vert{(\Sigma_k)^{\frac{1}{2}}\tau}\Vert-\tau^{\top}\mu_{k}+\xi_{k}$, $g_{k}(x)=x_{k}$, $k=1,...,K$, $h(x)=\log\hat{\epsilon}-\sum\limits_{k=1}^{K}x_{k}$, $\omega_{s}(\tau)=\sum\limits_{(s',a')\in\Lambda}\tau(s',a')\left(\delta(s,s')-\alpha p(s|s',a')\right)-(1-\alpha)q(s)$, $s\in S$ and $\nu(\tau)=-\tau$. We can then rewrite problem \eqref{399} as follows:
\begin{subequations}\label{37}
\begin{eqnarray}
& \min\limits_{\tau\in\mathbb{R}_{+}^{|\Lambda|}, x\in\mathbb{R}_{-}^{K}} & f(\tau) \\
&{\rm s.t.} & \phi_{k}(\tau, x)\le 0, k=1,...,K,\\
&& g_{k}(x)\le 0, k=1,...,K,\\
&& h(x)\le 0,\\
&& \omega_{s}(\tau)\le 0, s\in S, \\
&& -\omega_{s}(\tau)\le 0, s\in S, \\
&& \nu(\tau)\le 0.
\end{eqnarray}
\end{subequations}

The feasible set of problem \eqref{37} is denoted by 
$$U=\left\{(\tau, x)\in \mathbb{R}_{+}^{|\Lambda|}\times \mathbb{R}_{-}^{K} | \phi_{k}(\tau, x)\le 0, g_{k}(x)\le 0,  k=1,...,K, h(x)\le 0, \omega_{s}(\tau)=0, s\in S, \nu(\tau)\le 0 \right\}.$$ Let $$U(\tau)=\left\{ x\in \mathbb{R}_{-}^{K} | \phi_{k}(\tau, x)\le 0, g_{k}(x)\le 0, k=1,...,K, h(x)\le 0 \right\}$$ and $$U(x)=\left\{ \tau\in \mathbb{R}_{+}^{|\Lambda|} | \phi_{k}(\tau, x)\le 0,  k=1,...,K, \omega_{s}(\tau)=0, s\in S, \nu(\tau)\le 0\right\}.$$ 

\begin{definition}
$(\tau^{*}, x^{*})$ is called a partial optimum of problem \eqref{37} if $f(\tau^{*})\le f(\tau), \forall \tau \in U(x^{*})$ and $x^{*}\in U(\tau^{*})$.
\end{definition}

\begin{lemma}[\cite{jiang2021partial}]
In biconvex optimization problems, there always exists a partial optimal solution.
\end{lemma}

Now we derive the KKT conditions of problem \eqref{37}. Let $(\tau^{*}, x^{*})\in\mathbb{R}_{+}^{|\Lambda|}\times \mathbb{R}_{-}^{K}$, if there exist $\beta_{k}^{(\tau)}$, $\beta_{k}^{(x)}$, $\chi_{k}$, $k=1,...,K,$ $\zeta$, $\theta_{1,s}$, $\theta_{2,s}$, $s\in S$ and $\varrho$, such that 
\begin{subequations}\label{000}
\begin{eqnarray}
&& \nabla f(\tau^{*})+\sum\limits_{k=1}^{K}\beta_{k}^{(\tau)}\nabla_{\tau}\phi_{k}(\tau^{*}, x^{*})+\sum\limits_{s\in S}(\theta_{1,s}-\theta_{2,s})\nabla\omega_{s}(\tau^{*})+\varrho\nabla\nu(\tau^{*})=0, \\
&& \sum\limits_{k=1}^{K}\beta_{k}^{(x)}\nabla_{x}\phi_{i}(\tau^{*}, x^{*})+\sum\limits_{k=1}^{K}\chi_{k}\nabla g_{k}(x^{*})+\zeta\nabla h(x^{*})=0, \\
&& \beta_{k}^{(\tau)}\ge 0, \ \ \beta_{k}^{(\tau)}\phi_{k}(\tau^{*}, x^{*})=0, \ \ k=1,...,K, \\ %%e 
&&  \beta_{k}^{(x)}\ge 0, \ \ \beta_{k}^{(x)}\phi_{k}(\tau^{*}, x^{*})=0, \ \ k=1,...,K,\\ %%g
&&  \chi_{k}\ge 0, \ \ \chi_{k} g_{k}(x^{*})=0, \ \ k=1,...,K,\\ %%g
&& \zeta\ge 0, \ \ \zeta h(x^{*})=0, \\
&& \theta_{1,s}\ge 0,\ \ \theta_{1,s}\omega_{s}(\tau^{*})=0, \ \ s\in S,\\
&& \theta_{2,s}\ge 0, \ \ \theta_{2,s}\omega_{s}(\tau^{*})=0,  \ \ s\in S,\\
&& \varrho\ge 0, \ \ \varrho\nu(\tau^{*})=0,%% h
\end{eqnarray}
\end{subequations} then $(\tau^{*}, x^{*})$ is called a partial KKT point of problem \eqref{37}, where $\beta_{k}^{(\tau)}$, $\beta_{k}^{(x)}$, $\chi_{k}$, $\zeta$, $\theta_{1,s}$, $\theta_{2,s}$ and $\varrho$ are the Lagrange multipliers associated with the constraints $\phi_{k}, g_k, h, \omega_{s}, -\omega_{s}, \nu$ with respect to the two sub problems with fixed $\tau$ or $x,$ respectively. Furthermore, we have the following lemma which gives the optimal KKT conditions of problem \eqref{37}.
\begin{lemma}[\cite{jiang2021partial}]\label{KKT}
    Let $\beta_{k}=\beta_{k}^{(\tau)}=\beta_{k}^{(x)}$, $\forall k=1,2,...,K$, then the partial optimum $(\tau^{*}, x^{*})$ is a KKT point of \eqref{37}. The KKT system of problem \eqref{37} is
\begin{subequations}
\begin{eqnarray}
&& \nabla f(\tau^{*})+\sum\limits_{k=1}^{K}\beta_{k}\nabla_{\tau}\phi_{k}(\tau^{*}, x^{*})+\sum\limits_{s\in S}(\theta_{1,s}-\theta_{2,s})\nabla\omega_{s}(\tau^{*})+\varrho\nabla\nu(\tau^{*})=0, \\
&& \sum\limits_{k=1}^{K}\beta_{k}\nabla_{x}\phi_{k}(\tau^{*}, x^{*})+\sum\limits_{k=1}^{K}\chi_{k}\nabla g_{k}(x^{*})+\zeta\nabla h(x^{*})=0, \\
&& \beta_{k}\ge 0, \ \ \beta_{k}\phi_{k}(\tau^{*}, x^{*})=0, \ \ \beta_{k}\phi_{k}(\tau^{*}, x^{*})=0, \ \ k=1,...,K, \\ %%e 
&&  \chi_{k}\ge 0, \ \ \chi_{k} g_{k}(x^{*})=0, \ \ k=1,...,K,\\ %%g
&& \zeta\ge 0, \ \ \zeta h(x^{*})=0, \\
&& \theta_{1,s}\ge 0,\ \ \theta_{1,s}\omega_{s}(\tau^{*})=0, \ \ \theta_{2,s}\ge 0, \ \ \theta_{2,s}\omega_{s}(\tau^{*})=0,\ \ s\in S,\\
&& \varrho\ge 0, \ \ \varrho\nu(\tau^{*})=0.%% h
\end{eqnarray}
\end{subequations}    
\end{lemma}

%Let $(\tau^{*}, x^{*})$ be a partial solution of \eqref{37} with respect to partial Slater constraints qualification at $(\tau^{*}, x^{*})$, then $(\tau^{*}, x^{*})$ is a partial optimum of \eqref{37} if and only if the partial KKT system \eqref{000} holds. Furthermore, if $\beta_{k}^{(\tau)}$=$\beta_{k}^{(x)}$, then $(\tau^{*}, x^{*})$ is a KKT point of \eqref{37}  \cite{jiang2021partial}.

\subsection{Neural network model}
In this section, we introduce a DNN model to solve problem \eqref{37}. 
Let $\phi(\tau, x)=(\phi_{1}(\tau,x), ..., \phi_{K}(\tau,x))^{\top},$ $g(x)=(g_{1}(x), ..., g_{K}(x))^{\top}$ and $\omega(\tau)=(\omega_{s}(\tau))^{\top}_{s\in S},$ we define the dynamical equations of the DNN as
\begin{equation}\label{ujm}
\begin{aligned}
\frac{d\tau}{dt}= & -\left(\nabla f(\tau)+\nabla_{\tau}\phi(\tau,x)^{\top}(\beta+\phi(\tau,x))^{+}+\nabla\omega(\tau)^{\top}(\theta_{1}+\omega(\tau))^{+}-\nabla\omega(\tau)^{\top}(\theta_{2}-\omega(\tau))^{+}\right. \\
& \left.+\nabla\nu(\tau)^{\top}(\varrho+\nu(\tau))^{+} \right),\\
\frac{dx}{dt}=& -\left(\nabla_{x}\phi(\tau,x)^{\top}(\beta+\phi(\tau,x))^{+}+\nabla g(x)^{\top}(\chi+g(x))^{+}+\nabla h(x)^{\top}(\zeta+h(x))^{+}\right),\\
 \frac{d\beta}{dt}=& (\beta+\phi(\tau, x))^{+}-\beta,\\
 \frac{d\chi}{dt}=& (\chi+g(x))^{+}-\chi,\\
 \frac{d\zeta}{dt}=& (\zeta+h(x))^{+}-\zeta,\\
 \frac{d\theta_1}{dt}=& (\theta_1+\omega(\tau))^{+}-\theta_1,\\
 \frac{d\theta_2}{dt}=& (\theta_2-\omega(\tau))^{+}-\theta_2,\\
 \frac{d\varrho}{dt}=& (\varrho+\nu(\tau))^{+}-\varrho.
\end{aligned}
\end{equation}

Let $z=(\tau, x, \beta, \chi, \zeta, \theta_1, \theta_2, \varrho)$, then the dynamical system can be written as 
\begin{equation}\label{27}
\left\{\begin{array}{l}
\frac{dz}{dt}=\kappa\varphi(z), \\
z(t_0)=z_0, 
\end{array}\right.
\end{equation} where
$$\varphi(z)=
\left[\begin{array}{l}
\varphi_{1}(z)\\
\varphi_{2}(z)\\
\varphi_{3}(z)\\
\varphi_{4}(z)\\
\varphi_{5}(z)\\
\varphi_{6}(z)\\
\varphi_{7}(z)\\
\varphi_{8}(z)
\end{array}\right]=
\left[\begin{array}{l}
-\big(\nabla f(\tau)+\nabla_{\tau}\phi(\tau,x)^{\top}(\beta+\phi(\tau,x))^{+}+
%\nabla_{\tau}g(\tau,x)^{\top}(\chi+g(\tau,x))^{+}+\nabla_{\tau}h(\tau,x)^{\top}(\zeta+h(\tau,x))^{+}  \\ 
\nabla\omega(\tau)^{\top}(\theta_{1}+\omega(\tau))^{+}-\nabla\omega(\tau)^{\top}(\theta_{2}-\omega(\tau))^{+}\\ +\nabla\nu(\tau)^{\top}(\varrho+\nu(\tau))^{+} \big) \\
-\big(\nabla_{x}\phi(\tau,x)^{\top}(\beta+\phi(\tau,x))^{+}+\nabla g(x)^{\top}(\chi+g(x))^{+}+\nabla h(x)^{\top}(\zeta+h(x))^{+}\big)\\
(\beta+\phi(\tau, x))^{+}-\beta\\
(\chi+g(x))^{+}-\chi\\
(\zeta+h(x))^{+}-\zeta\\
(\theta_1+\omega(\tau))^{+}-\theta_1\\
(\theta_2-\omega(\tau))^{+}-\theta_2\\
(\varrho+\nu(\tau))^{+}-\varrho
\end{array}\right],
$$ $z_0$ is a given initial point and $\kappa$ is a scale parameter that indicates the convergence rate of the DNN model \eqref{ujm}. For the sake of simplicity, we set $\kappa=1$. Figure 2 shows the structure of our proposed DNN model,  {where the frame in the middle illustrates the cyclic structure of the neural network. For example, the arrow from $\dot{\tau}=\varphi_{1}$ to $\dot{\theta_1}=\varphi_6$ means that 
%when 
the in-process solution $\tilde{\tau}$ which is generated by $\dot{\tau}=\varphi_{1}$ %, it then 
will be used to update $\theta_{1}$ by  $\dot{\theta_1}=\varphi_6$.
%to help generate $\theta_{1}$. 
These arrows among eight differential equations constitute the cyclic structure of the neural network.}
\begin{figure}[htb]
    \centering
    \includegraphics[width=12cm]{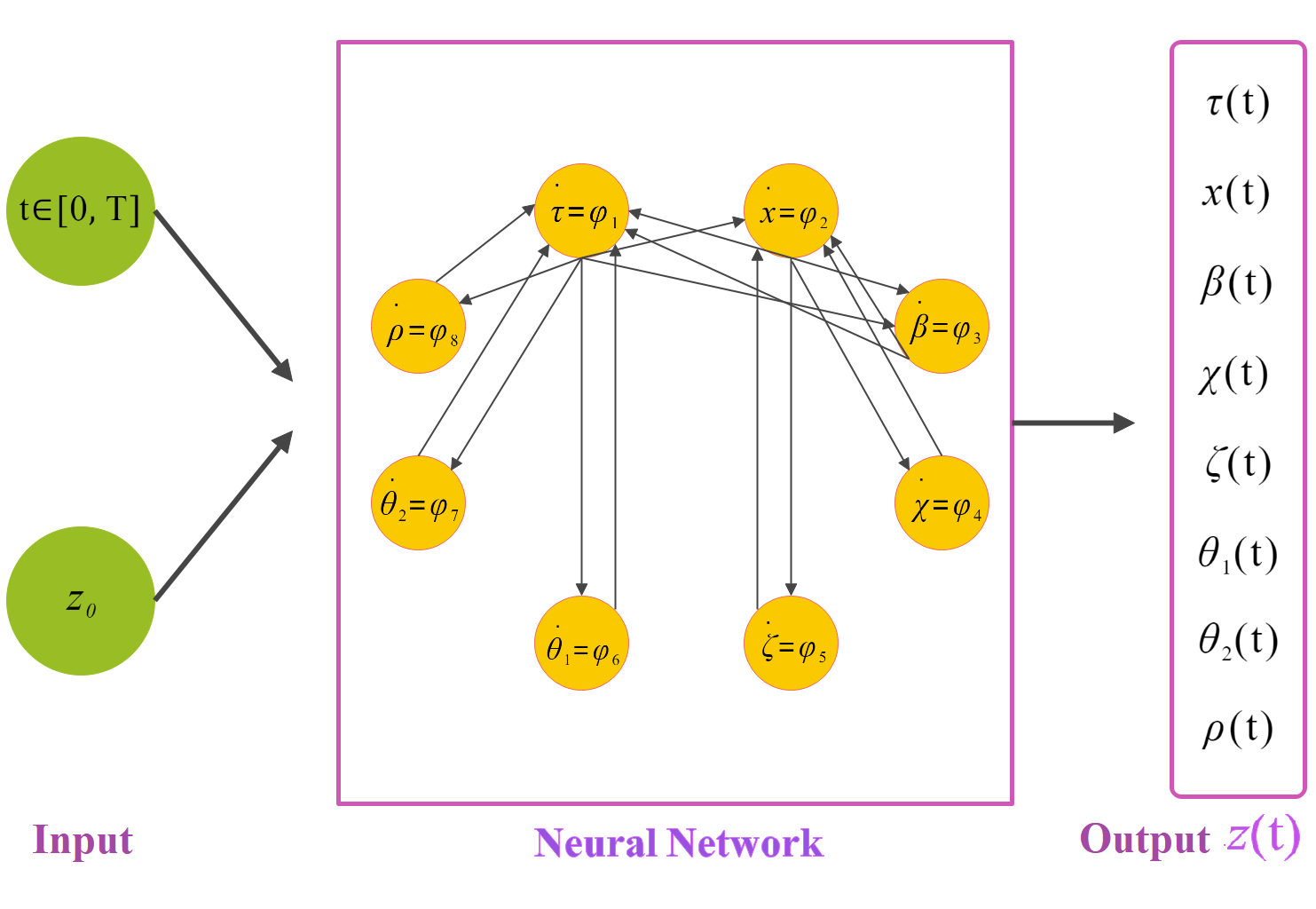}
    \caption{The structure of DNN model}
    
\end{figure}

\begin{theorem}\label{4}
Let $(\tau^*, x^*, \beta^*, \chi^*, \zeta^*, \theta_1^*, \theta_2^*, \varrho^*)$ be an equilibrium point of the neural network \eqref{ujm}, then $(\tau^*, x^*)$ is a KKT point of problem \eqref{37}. On the other hand, if $(\tau^*, x^*)$ is a KKT point of problem \eqref{37}, then there exist $\tau^*\ge 0, x^*\ge 0, \beta^*\ge 0, \chi^*\ge 0, \zeta^*\ge 0, \theta_1^*\ge 0, \theta_2^*\ge 0, \varrho^*\ge 0$ such that $(\tau^*, x^*, \beta^*, \chi^*, \zeta^*, \theta_1^*, \theta_2^*, \varrho^*)$ is an equilibrium point of the DNN model \eqref{27}.
\end{theorem}

\begin{proof}
Let $(\tau^*, x^*, \beta^*, \chi^*, \zeta^*, \theta_1^*, \theta_2^*, \varrho^*)$ be an equilibrium point of the neural network \eqref{ujm}, then $\frac{d\tau^*}{dt}=\frac{dx^*}{dt}=\frac{d\beta^*}{dt}=\frac{d\chi^*}{dt}=\frac{d\zeta^*}{dt}=\frac{d\theta_1^*}{dt}=\frac{d\theta_2^*}{dt}=\frac{d\varrho^*}{dt}=0$. Thus we have 
 $$
 \begin{aligned}
 & -\left(\nabla f(\tau^*)+\nabla_{\tau}\phi(\tau^*,x^*)^{\top}(\beta^*+\phi(\tau^*,x^*))^{+}+\nabla\omega(\tau^*)^{\top}(\theta_{1}^*+\omega(\tau^*))^{+}-\nabla\omega(\tau^*)^{\top}(\theta_{2}^*-\omega(\tau^*))^{+}\right.\\
 &\qquad \left.+\nabla\nu(\tau^*)^{\top}(\varrho^*+\nu(\tau^*))^{+} \right)=0,\\
 & -\big(\nabla_{x}\phi(\tau^*,x^*)^{\top}(\beta^*+\phi(\tau^*,x^*))^{+}+\nabla g(x^*)^{\top}(\chi^*+g(x^*))^{+}+\nabla h(x^*)^{\top}(\zeta^*+h(x^*))^{+}\big)=0,\\
& (\beta^*+\phi(\tau^*, x^*))^{+}-\beta^*=0, \ \ (\chi^*+g(x^*))^{+}-\chi^*=0, \ \  (\zeta^*+h(x^*))^{+}-\zeta^*=0,\\
& (\theta_1^*+\omega(\tau^*))^{+}-\theta_1^*=0, \ \ (\theta_2^*-\omega(\tau^*))^{+}-\theta_2^*=0, \ \ (\varrho^*+\nu(\tau^*))^{+}-\varrho^*=0.
 \end{aligned}
 $$ Notice that $(\beta^*+\phi(\tau^*, x^*))^{+}-\beta^*=0$ if and only if $\beta^{*}\ge 0, \phi(\tau^*, x^*)\le 0$ and ${\beta^{*}}^{\top}\phi(\tau^{*},x^{*})=0$. Following the same way, we have $\chi^{*}\ge 0, g(\tau^*, x^*)\le 0$, ${\chi^{*}}^{\top}g(x^{*})=0$, $\zeta^{*}\ge 0, h( x^*)\le 0$, ${\zeta^{*}}^{\top}h(x^{*})=0$, $\theta_{1}^{*}\ge 0, \omega(\tau^*)\le 0$, ${\theta_{1}^{*}}^{\top}\omega(\tau^{*})=0$, $\theta_{2}^{*}\ge 0, -\omega(\tau^*)\le 0$, ${\theta_{2}^{*}}^{\top}\omega(\tau^{*})=0$, $\varrho^{*}\ge 0, \nu(\tau^*)\le 0$, ${\varrho^{*}}^{\top}\nu(\tau^{*})=0$, which implicates \eqref{ujm}. Thus we obtain the partial KKT system \eqref{000} with $\beta^{(\tau)}$=$\beta^{(x)}=\beta^{*}$. By Lemma \ref{KKT}, $(\tau^{*}, x^{*})$ is the KKT point. The converse conclusion is straightforward.
\end{proof}

\begin{theorem}
For any initial point $z_{0}=(\tau_0, x_0, \chi_0, \zeta_0, \theta_{1}^{0}, \theta_{2}^{0}, \varrho_0)$, there exists a unique continuous solution $z(t)=(\tau(t), x(t), \chi(t), \zeta(t), \theta_{1}(t), \theta_{2}(t), \varrho(t))$ for \eqref{ujm}.
\end{theorem}
\begin{proof}
As $\nabla f(\tau), \nabla_{\tau}\phi_{k}(\tau, x), \nabla_{x}\phi_{k}(\tau, x), \nabla g_{k}(x), \nabla h(x), \nabla\omega_{s}(\tau), \nabla\nu(\tau)$ are all continuously differentiable, then all terms of \eqref{ujm} are locally Lipschitz continuous. {Therefore $\varphi(z)$ is locally Lipschitz continuous with respect to $z$.} For any given closed interval $[0, M_1]$, by the local existence theorem of ODE (Picard–Lindelöf Theorem \cite{hu2005theory}), the neural network \eqref{ujm} has a unique continuous solution when $t\in [0,M_1]$. As for any bounded closed interval, $z$ is bounded. Then we know from the continuation theorem \cite{hu2005theory} that the unique continuous solution in $[0,M_1]$ can be extended to interval $[0,+\infty)$, which completes this proof.
\end{proof}

\subsection{Stability analysis}
In this section, we study the stability of the proposed DNN model. To prove our DNN model's stability and convergence, we first show that the Jacobian matrix $\nabla \varphi(z)$ is negative semidefinite.

\begin{lemma}\label{i9}
    The Jacobian matrix $\nabla \varphi(z)$ is a negative semidefinite matrix. 
\end{lemma}

\begin{proof}
For any $z$, we assume that there exist index sets $K_1,K_2,S_1,S_2$ such that
$$(\beta_k+\phi_k)^{+}=\left\{\begin{array}{l}
\beta_{k}+\phi_{k}(\tau,x),  k\in K_{1}, \\
0, k\notin K_{1},
\end{array}\right.
(\chi_k+g_k)^{+}=\left\{\begin{array}{l}
\chi_{k}+g_{k}(x),  k\in K_{2}, \\
0, k\notin K_{2},
\end{array}\right.
$$
$$({\theta_{1,s}}+\omega_s)^{+}=\left\{\begin{array}{l}
\theta_{1,s}+\omega_{s}(\tau),  s\in S_{1}, \\
0, s\notin S_{1},
\end{array}\right.
({\theta_{2,s}}-\omega_s)^{+}=\left\{\begin{array}{l}
\theta_{2,s}-\omega_{s}(\tau),  s\in S_{2}, \\
0, s\notin S_{2}.
\end{array}\right.
$$
Then we have $(\beta+\phi)^{+}=\left\{(\beta_k+\phi_k)^{+}\right\}_{k=1,2,...,K}$, $(\chi+g)^{+}=\left\{(\chi_k+g_k)^{+}\right\}_{k=1,2,...,K}$, $(\theta_1+\omega)^{+}=\left\{(\theta_{1,s}+\omega_s)^{+}\right\}_{s\in S}$ and $(\theta_2-\omega)^{+}=\left\{(\theta_{2,s}-\omega_s)^{+}\right\}_{s\in S}$.
For the sake of simplicity, we define $\phi^{K_1}(\tau,x)=(\phi_{k}(\tau,x))_{k\in K_1}$, $(\beta+\phi)^{K_1}=(\beta_k+\phi_k)_{k\in K_1}$, $g^{K_2}(x)=(g_{k}(x))_{k\in K_2}$, $(\chi+g)^{K_2}=(\chi_{k}+g_{k})_{k\in K_2}$, $\omega^{S_1}(\tau)=(\omega_{s}(\tau))_{s\in S_1}$, $(\theta_1+\omega)^{S_1}=(\theta_{1,s}+\omega_{s})_{s\in S_1}$, $\omega^{S_2}(\tau)=(\omega_{s}(\tau))_{s\in S_2}$ and $(\theta_2-\omega)^{S_2}=(\theta_{2,s}-\omega_{s})_{s\in S_2}$. Without loss of generality, we consider the case with $\zeta+h(x)\neq 0, \varrho+\nu(\tau)\neq 0$. We represent the Jacobian matrix $\nabla \varphi$ as follows
\begin{equation}
\nabla \varphi(z)=\left[\begin{array}{llllllll}
A_1 & A_2 & A_3 & A_4 & A_5 & A_6 & A_7 & A_8 \\
B_1 & B_2 & B_3 & B_4 & B_5 & B_6 & B_7 & B_8 \\
C_1 & C_2 & C_3 & C_4 & C_5 & C_6 & C_7 & C_8 \\
D_1 & D_2 & D_3 & D_4 & D_5 & D_6 & D_7 & D_8 \\
E_1 & E_2 & E_3 & E_4 & E_5 & E_6 & E_7 & E_8 \\
F_1 & F_2 & F_3 & F_4 & F_5 & F_6 & F_7 & F_8 \\
G_1 & G_2 & G_3 & G_4 & G_5 & G_6 & G_7 & G_8 \\
H_1 & H_2 & H_3 & H_4 & H_5 & H_6 & H_7 & H_8 
\end{array}\right],
\end{equation} where

$A_1=-(\nabla^{2} f(\tau)+\nabla_{\tau}^{2}\phi^{K_1}(\tau,x)^{\top}(\beta+\phi)^{K_1}+\nabla_{\tau}\phi^{K_1}(\tau,x)^{\top}\nabla_{\tau}\phi^{K_1}(\tau,x)+\nabla^{2}\omega^{S_1}(\tau)^{\top}(\theta_{1}+\omega)^{S_1}+\nabla \omega^{S_1}(\tau)^{\top}\nabla\omega^{S_1}(\tau)-\nabla^{2}\omega^{S_2}(\tau)^{\top}(\theta_{2}-\omega)^{S_2}+\nabla\omega^{S_2}(\tau)^{\top}\nabla\omega^{S_2}(\tau)+\nabla^{2}\nu(\tau)(\varrho+\nu)+\nabla\nu(\tau)^{\top}\nabla\nu(\tau)),$\\

$A_2=-(\nabla_{x}\nabla_{\tau}\phi^{K_1}(\tau,x)^{\top}(\beta+\phi)^{K_1}+\nabla_{x}\phi^{K_1}(\tau,x)^{\top}\nabla_{\tau}\phi^{K_1}(\tau,x)),$\\

$B_1=-(\nabla_{\tau}\nabla_{x}\phi^{K_1}(\tau,x)^{\top}(\beta+\phi)^{K_1}+\nabla_{\tau}\phi^{K_1}(\tau,x)^{\top}\nabla_{x}\phi^{K_1}(\tau,x)),$\\

$B_2=-(\nabla_{x}^{2}\phi^{K_1}(\tau,x)^{\top}(\beta+\phi)^{K_1}+\nabla_{x}\phi^{K_1}(\tau,x)^{\top}\nabla_{x}\phi^{K_1}(\tau,x)+\nabla^{2}g^{K_2}(x)^{\top}(\chi+g)^{K_2}+\nabla g^{K_2}(x)^{\top}\nabla g^{K_2}(x)+\nabla^{2}h(x)(\zeta+h)+\nabla h(x)^{\top}\nabla h(x)),$\\

$A_{3}=-C_{1}=-\nabla_{\tau}\phi^{K_1}(\tau,x)^{\top},$ \ \ $A_{4}=-D_{1}=0,$ \ \ $A_{5}=-E_{1}=0,$

$A_{6}=-F_{1}=-\nabla\omega^{S_1}(\tau)^{\top},$ \ \ $A_{7}=-G_{1}=\nabla\omega^{S_2}(\tau)^{\top},$ \ \ $A_{8}=-H_{1}=-\nabla\nu(\tau)^{\top},$

$B_{3}=-C_{2}=-\nabla_{x}\phi^{K_1}(\tau,x)^{\top},$ \ \ $B_{4}=-D_{2}=-\nabla g^{K_2}(x)^{\top},$ \ \ $B_{5}=-E_{2}=-\nabla h(x)^{\top},$

$B_{6}=-F_{2}=0,$ \ \ $B_{7}=-G_{2}=0,$ \ \ $B_{8}=-H_{2}=0,$\\

$
\begin{aligned}
C_{3}=\left[\begin{array}{ccccc}
\ddots & & & & 0 \\
& -1 & & & \\
 &  & \ddots & & \\
 &  & & -1 & \\
0 & & & & \ddots
\end{array}\right],\ \ 
D_{4}=\left[\begin{array}{ccccc}
\ddots & & & & 0 \\
& -1 & & & \\
 &  & \ddots & & \\
 &  & & -1 & \\
0 & & & & \ddots
\end{array}\right],
\end{aligned}
$ \\
%\ \ D_{4}=\left[\begin{array}{ll}
%\boldsymbol{O}_{K_2\times K_2} & \boldsymbol{O}_{K_2\times (K-K_2)} \\
%\boldsymbol{O}_{(K-K_2)\times K_2} & -\boldsymbol{I}_{(K-K_2)\times (K-K_2)}
%\end{array}\right],$

$
\begin{aligned}
F_{6}=\left[\begin{array}{ccccc}
\ddots & & & & 0 \\
& -1 & & & \\
 &  & \ddots & & \\
 &  & & -1 & \\
0 & & & & \ddots
\end{array}\right],\ \ 
G_{7}=\left[\begin{array}{ccccc}
\ddots & & & & 0 \\
& -1 & & & \\
 &  & \ddots & & \\
 &  & & -1 & \\
0 & & & & \ddots
\end{array}\right],
\end{aligned}
$ 

%$F_{6}=\left[\begin{array}{ll}
%\boldsymbol{O}_{S_{1}\times S_{1}} & \boldsymbol{O}_{S_{1}\times (|S|-S_{1})} \\
%\boldsymbol{O}_{(|S|-S_{1})\times S_{1}} & -\boldsymbol{I}_{(|S|-S_{1})\times (|S|-S_{1})}
%\end{array}\right], \ \ G_{7}=\left[\begin{array}{ll}
%\boldsymbol{O}_{S_{2}\times S_{2}} & \boldsymbol{O}_{S_{2}\times (|S|-S_{2})} \\
%\boldsymbol{O}_{(|S|-S_{2})\times S_{2}} & -\boldsymbol{I}_{(|S|-S_{2})\times (|S|-S_{2})}
%\end{array}\right],$\\

$C_{4}=0, C_{5}=0, C_{6}=0, C_{7}=0, C_{8}=0, D_{3}=0, D_{5}=0, D_{6}=0, D_{7}=0, D_{8}=0, F_{3}=0, $

$F_{4}=0, F_{5}=0, F_{7}=0, F_{8}=0, G_{3}=0, G_{4}=0, G_{5}=0, G_{6}=0, G_{8}=0,$ \\
where $C_{3}(k,k)=-1$ if $k\in K_1$, $D_{4}(k,k)=-1$ if $k\in K_2$, $F_{6}(s,s)=-1$ if $s\in S_1$, $G_{7}(s,s)=-1$ if $s\in S_2$ and all other elements in $C_3, D_4, F_6, G_7$ are equal to $0$.

Since $\phi$ is twice differentiable, by Schwarz's theorem, we have $\nabla_{\tau}\nabla_{x}\phi^{K_1}(\tau, x)=\nabla_{x}\nabla_{\tau}\phi^{K_1}(\tau, x)$. Then we get $A_{2}=B_{1}^{\top}$, so $\nabla \varphi(z)$ becomes

\begin{equation}
\nabla \varphi(z)=\left[\begin{array}{llllllll}
A_1 & B_1^{\top} & A_3 & 0 & 0 & A_6 & A_7 & A_8 \\
B_1 & B_2 & B_3 & B_4 & B_5 & 0 & 0 & 0 \\
-A_3 & -B_3 & C_{3} & 0 & 0 & 0 & 0 & 0 \\
0 & -B_4 & 0 & D_{4} & 0 & 0 & 0 & 0 \\
0 & -B_5 & 0 & 0 & 0 & 0 & 0 & 0 \\
-A_6 & 0 & 0 & 0 & 0 & F_{6} & 0 & 0 \\
-A_7 & 0 & 0 & 0 & 0 & 0 & G_{7} & 0 \\
-A_8 & 0 & 0 & 0 & 0 & 0 & 0 & 0 
\end{array}\right].
\end{equation}

It is clear that $C_3, D_4, F_6, G_7$ are negative semidefinite. Thus $$ \boldsymbol{S}=\left[\begin{array}{llllll}
C_{3} & 0 & 0 & 0 & 0 & 0 \\
0 & D_{4} & 0 & 0 & 0 & 0 \\
0 & 0 & 0 & 0 & 0 & 0 \\
0 & 0 & 0 & F_{6} & 0 & 0 \\
0 & 0 & 0 & 0 & G_{7} & 0 \\
0 & 0 & 0 & 0 & 0 & 0 
\end{array}\right]$$  is negative semidefinite. 

Since the function $f$ is convex and twice differentiable,  $\nabla^{2}f(\tau)$ is positive semidefinite. Furthermore, $\phi$ is biconvex, $g, h, \omega, \nu$ are convex and $\phi, g, h, \omega, \nu$ are all twice differentiable. Thus by \cite{gorski2007biconvex}, we have that $\nabla^{2}_{\tau}\phi^{K_1}(\tau, x), \nabla^{2}_{x}\phi^{K_1}(\tau, x)$, $\nabla^{2}g^{K_2}(x)$, $\nabla^{2}\omega^{S_1}(\tau)$, $\nabla^{2}\omega^{S_2}(\tau)$, $\nabla^{2}h(x), \nabla^{2}\nu(\tau)$ are all positive semidefinite. It follows that $A_1, B_2$ are negative semidefinite and then $A=\left[\begin{array}{ll}
A_1 & B_1^{\top} \\
B_1 & B_2
\end{array}\right]$ is negative semidefinite. 

Let $B=\left[\begin{array}{llllll}
A_3 & 0 & 0 & A_6 & A_7 & A_8 \\
B_3 & B_4 & B_5 & 0 & 0 & 0
\end{array}\right]$, then $\nabla \varphi(z)$ can be written as $\nabla \varphi(z)=\left[\begin{array}{ll}
A & B \\
-B^{\top} & \boldsymbol{S}
\end{array}\right]$. Since $A, \boldsymbol{S}$ are both negative semidefinite, it is then known from \cite{foias1990positive} that $\nabla \varphi(z)$ is negative semidefinite. This establishes the desired result. The proof when $\zeta+h(x)=0$ or $\varrho+\nu(\tau)=0$ follows similar lines.
\end{proof}

Before showing the stability of the DNN model, we introduce the following definition and lemma.

\begin{definition}[\cite{rockafellar2009variational}]
    A mapping $F: \mathbb{R}^n\rightarrow \mathbb{R}^n$ is said to be monotonic if $$(x-y)^{\top}(F(x)-F(y))\ge 0, \forall x,y\in\mathbb{R}^n.$$
\end{definition}

\begin{lemma}[\cite{rockafellar2009variational}]\label{i8}
    A differentiable mapping $F:\mathbb{R}^n\rightarrow \mathbb{R}^n$ is monotonic if and only if the Jacobian matrix $\nabla F(x), \forall x\in \mathbb{R}^n$ is positive semidefinite.
\end{lemma}

\begin{theorem}\label{55}
The DNN model \eqref{ujm} is stable in the Lyapunov sense and converges to $(\tau^*, x^*, \beta^*, \chi^*, \zeta^*, \theta_1^*, \theta_2^*, \varrho^*)$, where $(\tau^*, x^*)$ is a KKT point of problem \eqref{37}.
\end{theorem}

\begin{proof}
Let $z^{*}=(\tau^*, x^*, \beta^*, \chi^*, \zeta^*, \theta_1^*, \theta_2^*, \varrho^*)$ be an equilibrium point of \eqref{ujm}. We define the following Lyapunov function:
$$V(z)=\Vert \varphi(z)\Vert^2+\frac{1}{2}\Vert z-z^{*}\Vert^2.$$ We have that 
$$\frac{dV(z(t))}{dt}=(\frac{d\varphi}{dt})^{\top}\varphi+\varphi^{\top}\frac{d\varphi}{dt}+(z-z^*)^{\top}\frac{dz}{dt}.$$ Since $\frac{dz}{dt}=\varphi (z)$ by \eqref{27}, we have $\frac{d\varphi}{dt}=\frac{d \varphi}{d z}\frac{dz}{dt}=\nabla \varphi(z)\varphi(z)$ and thus $$\frac{dV(z(t))}{dt}=\varphi^{\top}(\nabla \varphi(z)^{\top}+\nabla \varphi(z))\varphi+(z-z^*)^{\top}\varphi(z).$$ By Lemma \ref{i9}, $\nabla \varphi$ is negative semidefinite, then we have $\varphi^{\top}(\nabla \varphi(y)^{\top})\varphi\le 0$, $\varphi^{\top}(\nabla \varphi(y))\varphi\le 0$. By Theorem \ref{4}, the equilibrium point $z^{*}$ satisfies $\varphi(z^{*})=\frac{dz^*}{dt}=0$. Then by Lemma \ref{i9} and Lemma \ref{i8}, $-\varphi$ is monotonic. Therefore, we have $(z-z^*)^{\top}\varphi(z)=(z-z^*)^{\top}(\varphi(z)-\varphi(z^*))\le 0$. This means that $\frac{dV(z(t))}{dt}\le 0$. Therefore the DNN model \eqref{ujm} is stable at $z^*$ in the sense of Lyapunov.

Since $V(z)\ge \frac{1}{2}\Vert z-z^*\Vert^2$, there exists a convergent subsequence $(z(t_k)_{k\ge 0})$ such that $\lim\limits_{k\rightarrow \infty}z(t_k)=\hat{z}$. Let $X=\{z(t)|\frac{dV(z(t))}{dt}=0\}$. By \cite{slotine1991applied}, using LaSalle's invariance principle, we have that any solution starting from a given $z_0$ converges to the largest invariant set contained in $X$. Notice that $$\left\{\begin{array}{l}
\frac{d\tau}{dt}=0 \\
\frac{dx}{dt}=0 \\
\frac{d\beta}{dt}=0 \\
\frac{d\chi}{dt}=0 \\
\frac{d\zeta}{dt}=0 \\
\frac{d\theta_1}{dt}=0 \\
\frac{d\theta_2}{dt}=0 \\
\frac{d\varrho}{dt}=0 \\
\end{array} \Leftrightarrow \frac{dV(z(t))}{dt}=0.\right.$$ It follows that $\hat{z}$ is an equilibrium point of \eqref{ujm}.

Next we consider a new Lyapunov function defined by $\hat{V}(z)=\Vert \varphi(z)\Vert^2+\frac{1}{2}\Vert z-\hat{z}\Vert^2$. Since $\hat{V}$ is continuously differentiable, $\hat{V}(\hat{z})=0$ and $\lim\limits_{k\rightarrow\infty}z(t_{k})=\hat{z}$, then we have $\lim\limits_{t\rightarrow\infty}\hat{V}(z(t))=\hat{V}(\hat{z})=0$. Furthermore, by $\hat{V}(z(t))\ge \frac{1}{2}\Vert z-\hat{z}\Vert^{2}$, we have $\lim\limits_{t\rightarrow \infty}\Vert z-\hat{z}\Vert=0$ and $\lim\limits_{t\rightarrow \infty}z(t)=\hat{z}$. It can then be deduced that the DNN model \eqref{ujm} is convergent in the sense of Lyapunov to an equilibrium point $\hat{z}=(\hat{\tau}, \hat{x}, \hat{\beta}, \hat{\chi}, \hat{\zeta}, \hat{\theta_1}, \hat{\theta_2}, \hat{\varrho})$ where $(\hat{\tau}, \hat{x})$ is a KKT point of \eqref{37}. Then we finish the proof.

\end{proof}

\section{Numerical experiments}\label{numerical}
In Section \ref{es}, we give the setup for the numerical experiments. %and introduce the fundamental parameters, software and hardware environment. 
In Sections \ref{num2}-\ref{num5}, we will carry out a series of numerical tests to compare the DNN approach and the SCA algorithm, including optimal policy results, convergence results, accuracy results and the generalization ability.  {In Section \ref{cne}, we give a conclusion of the numerical experiments.}
%We also present our findings in each section.
%In Section \ref{cne}, we conclude some findings from the numerical experiments.

\subsection{Experimental setup}\label{es}
In order to evaluate the performance of our DNN approach, we consider the machine replacement problem, which was studied in \cite{delage2010percentile,varagapriya2022constrained,wiesemann2013robust,goyal2022robust,ramani2022robust}. In the machine replacement problem, we consider an opportunity cost $c_0$ in the objective function along with two kinds of maintenance costs $c_1$ and $c_2$ in the constraint. The opportunity cost $c_0$ comes from the potential production losses when the machine is under repair. The maintenance cost comes from two parts: $c_1$ is the operational consumption of machines, such as the required electricity fee and fuel costs when the machine is working; $c_2$ is the cost occurred from production of inferior quality products. 
%These three costs are unknown in advance and we just know the mean values and corresponding covariance matrix for each cost. Suppose the owner possesses a fixed number of the same machines, and we assume that each machine is modeled with the same model. Therefore we only consider one machine and the same repair policy for it can be applied uniformly for all the machines.
We set the states as the used age of a machine. At each state,  there are two possible actions: $a_1$ is to repair and $a_2$ is not to repair. Three considered costs introduced above are incurred at each state. We assume that there are 5 states in all. The transition probabilities are given in Figure 3, where the solid lines denote the case when the decision maker chooses the action “repair" and the dashed lines denote the case when the decision maker chooses the action “do not repair". The numbers along with the lines denote the transition probabilities. For example, at states $2$ - $5$, if the repair action is taken, the machine moves to state 1 with probability 0.8 and stays in the current state with probability 0.2. At state $1$, it stays in the state with probability 1 under the repair action. If the “do not repair" action is taken, the machine moves to the next state with probability 0.9 and stays in the current state with probability 0.1; whereas in the last state, it remains there with probability 1. For the above problem, we consider the corresponding moment-based J-DRCCMDP problem \eqref{399}. The objective of the MDP is to maximize the discounted value function (minimizing the corresponding cost function), while keeping the two maintenance costs not exceeding the tolerance levels with a large probability. 

\begin{figure}[htb]
    \centering
    \includegraphics[width=9cm]{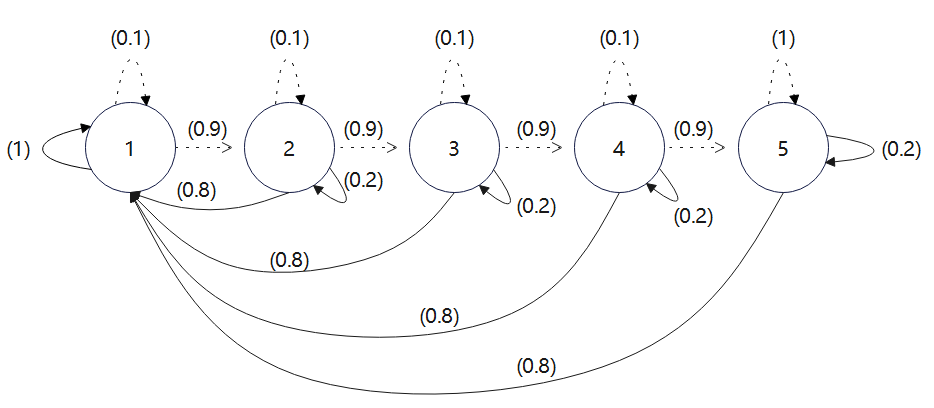}
    \caption{The transition probabilities for the MDP}
    \label{fig:my_label}
\end{figure}

For the MDP, we choose the discount factor $\alpha=0.6$ and assume that the initial distribution $q$ is a uniform distribution. The mean values of the three costs are shown in Table 1. For example, at state 1, if the “repair” action $a_1$ is taken, the mean values of 
$c_0, c_1, c_2$ are $1, 1.5, 0$,
%$r_0, r_1, r_2$ are $-1, -1.5, 0$, 
respectively; if the “do not repair” action $a_2$ is taken, the mean values of three costs are $0, 8, 5,$
%$0, -8, -5$ 
respectively. The last two states are risky states for which the mean values of costs are much larger. We assume that the reference covariance matrices of the three costs are diagonal and positive definite. Concretely, for both actions, the reference covariance matrix of $r_0, r_1, r_2$ are $\Sigma_0=diag([0.3,0.3,0.3,0.3,0.3,0.3,5,2,8,9])$, $\Sigma_1=diag([0.5,0.5,0.5,0.5,0.5,0.5,8,9,8,9])$ and $\Sigma_2=diag([0.4,0.4,0.4,0.4,0.4,0.4,9,8,8.5,10]),$ respectively.  In our experiment, We set $\rho_{1,0}=\rho_{1,1}=\rho_{2,1}=0.1$, $\rho_{1,2}=\rho_{2,2}=0.15$, $\xi_{1}=\xi_{2}=-40$ and $\hat{\epsilon}=0.95$.

% \usepackage{multirow}
\begin{table}[!ht]\label{yu}
\setlength{\belowcaptionskip}{0.2cm}
\caption{\normalsize The mean values of three costs}
\centering
\begin{tabular}{c|c|c|c|c|c|c}
\hline
\multirow{2}{*}{States} & \multicolumn{2}{c|}{Maintenance cost}     & \multicolumn{2}{c|}{Operation consumption cost}     & \multicolumn{2}{c}{Inferior quality cost}      \\ \cline{2-7} 
                   & {$c_0(s,a_1)$} & $c
                   %r
                   _0(s,a_2)$ & {$c_1(s,a_1)$} & $c_1(s,a_2)$ & {$c_2(s,a_1)$} & $c_2(s,a_2)$ \\ \hline
1                  & \multicolumn{1}{c|}{1} & 0 & \multicolumn{1}{c|}{1.5} & 8 & {0} & 5 \\ \hline
2                  & \multicolumn{1}{c|}{1} & 0 & \multicolumn{1}{c|}{1.5} & 8 & {0} & 5 \\ \hline
3                  & \multicolumn{1}{c|}{1} & 0 & \multicolumn{1}{c|}{1.5} & 8 & {0} & 8 \\ \hline
4                  & \multicolumn{1}{c|}{4} & 30 & \multicolumn{1}{c|}{5} & 100 & {1.5} & 30 \\ \hline
5                  & \multicolumn{1}{c|}{4} & 70 & \multicolumn{1}{c|}{5} & 200 & {3} & 50 \\ \hline
\end{tabular}
\end{table} 

All numerical experiments are run on a PC with AMD Ryzen 7 5800H CPU and 16.0 GB RAM. We use Python 3.11 as our programming language to implement our DNN model. We solve the ODE systems corresponding to the DNN model with solve$\_$ivp from the scipy.integrate package. We notice that the complexity of solving the ODE system depends on the number of variables. In our test, the numbers of variables $\tau, x, \beta, \chi, \zeta, \theta_1, \theta_2, \varrho_0$ are $10, 2, 2, 2, 1, 5, 5, 10,$ respectively, and so 37 in total. We set the initial point $z_0=(\tau_0, x_0, \beta_0, \chi_0, \zeta_0, {\theta_1}_0, {\theta_2}_0, \varrho_0)$ as $(10^{-3}, 10^{-3}, 10^{-3}, 10^{-3}, 10^{-3}, 10^{-3}, 10^{-3}, 10^{-3}, 10^{-3}, 10^{-3}, -8, -60,
10^{-4}, 10^{-4}, 10^{-4}, 10^{-4}, 10^{-4}, 10^{-4}, 10^{-4},\\ 10^{-4}, 10^{-4}, 10^{-4}, 10^{-4}, 10^{-4}, 10^{-4}, 10^{-4}, 10^{-4}, 10^{-4}, 10^{-4}, 10^{-4}, 10^{-4}, 10^{-4}, 10^{-4}, 10^{-4}, 10^{-4}, 10^{-4}, 10^{-4})$. For comparison, we also use the sequential convex approximation (SCA) algorithm (Algorithm 1) to solve the problem \eqref{399}, where we set the initial points $h^{0}_{1}=0.83, h_{2}^{0}=0.85$, $n_{max}=100$, $\gamma=0.6$ and $L=10^{-8}$. We conduct Algorithm 1 in MATLAB and solve the sub-optimization problems by the MOSEK solver. The detailed numerical results are shown in the following 4 parts from different perspectives.

%\subsection{Numerical results}\label{nr}

\subsection{Optimal policy}\label{num2}
We solve the moment-based J-DRCCMDP problem \eqref{399} corresponding to the above machine replacement problem by using the DNN approach and SCA algorithm respectively, and especially derive the optimal solutions $\tau$ of problem \eqref{399}. Then by Remark \ref{TM}, we can compute the optimal policy of the machine replacement problem. For the DNN, we choose the optimal solution when $t=2000$. Table 2 shows the optimal policies at each state found by the two approaches. As there are only two actions at each state, the sum of probabilities of the two actions is 1, thus we denote the probability of one action $\approx1$ when the probability of the other action is very close to 0.

\begin{table}[!ht]\label{okl}
\centering
\caption{\normalsize Optimal policies of moment based J-DRCCMDP}
\begin{tabular}{cc|c|c|c|c|c}
\hline
\multicolumn{2}{c|}{State}                      & 1 & 2 & 3 & 4 & 5 \\ \hline
\multicolumn{1}{c|}{{DNN}} & repair & 1.7576e-08 & 2.8942e-08 & $\approx1$ & $\approx1$ & $\approx1$  \\ \cline{2-7} 
\multicolumn{1}{c|}{}                   & do not repair & $\approx1$ & $\approx1$ & 3.2052e-08 & 4.4931e-07 & 2.7283e-07 \\ \hline
\multicolumn{1}{c|}{{SCA}} & repair & 3.5698e-10 & 4.8275e-10 & $\approx1$ & $\approx1$ & $\approx1$ \\ \cline{2-7} 
\multicolumn{1}{c|}{}                   & do not repair & $\approx1$ & $\approx1$ & 2.1067e-11 & 2.2559e-10 & 5.0490e-10 \\ \hline
\end{tabular}
\end{table}

From Table 2, we can see that the optimal policy for each state derived by two approaches are nearly the same. For two approaches, the probability of “repair” at last three states are all $1$, which is consistent with the fact that the machine gets aging with the state moving forward.

\subsection{Convergence quality}
For the DNN model, we choose the interval $[0,2000]$ of $t$ to observe the convergence performance. Figure 4 shows the values of optimal solutions $\tau$ at the five states. For instance, in (a), the blue curve denotes the value of $\tau(1,a_1)$ and the orange curve denotes the value of $\tau(1,a_2)$, where $a_1, a_2$ correspond to the actions “repair" and “do not repair", respectively. We see from these figures that the solution $\tau$ converges at each state.
% \begin{figure}[h!] 
% \centering
% \subfloat[$m=n=50$, $100$ and $200$]{\includegraphics[scale=0.48]{t1.eps}}
% \subfloat[$m=n=200$ and $400$]{\includegraphics[scale=0.48]{t2.eps}}
% \caption{Numerical results obtained by the scheme with various spatial scales}
% \label{sec5figgrid}
% \end{figure}

\begin{figure}[htbp]
	\centering
	\begin{minipage}{0.31\linewidth}
		\centering
		\includegraphics[width=0.9\linewidth]{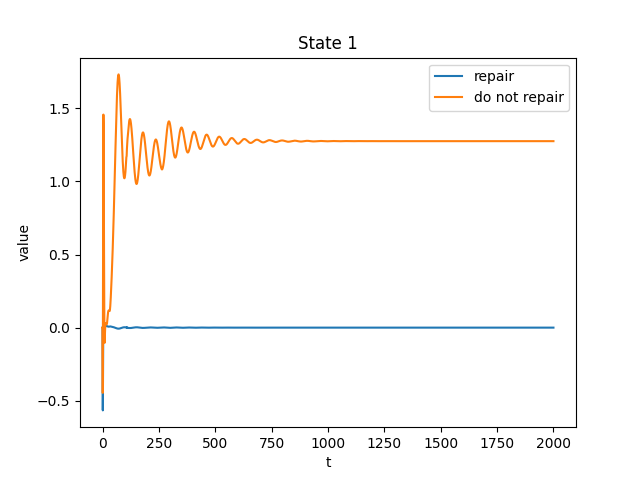}
		\caption{$\tau(1,a_1),\tau(1,a_2)$}
	\end{minipage}
	%\qquad
	\begin{minipage}{0.31\linewidth}
		\centering
		\includegraphics[width=0.9\linewidth]{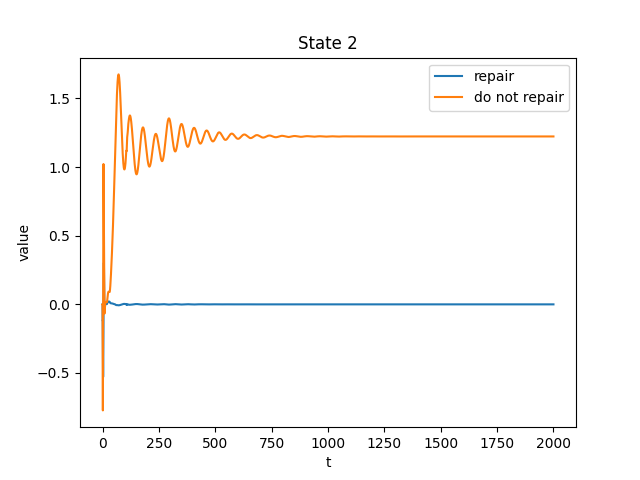}
		\caption{$\tau(2,a_1),\tau(2,a_2)$}
	\end{minipage}
 \begin{minipage}{0.31\linewidth}
        \centering
		\includegraphics[width=0.9\linewidth]{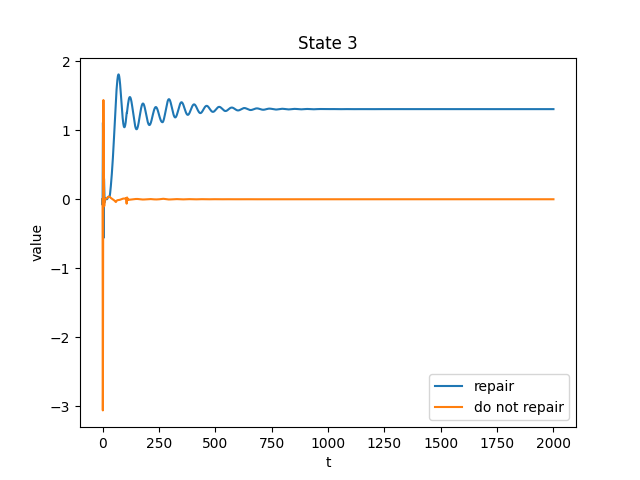}
		\caption{$\tau(3,a_1),\tau(3,a_2)$}
	\end{minipage}
 \begin{minipage}{0.31\linewidth}
        \centering
		\includegraphics[width=0.9\linewidth]{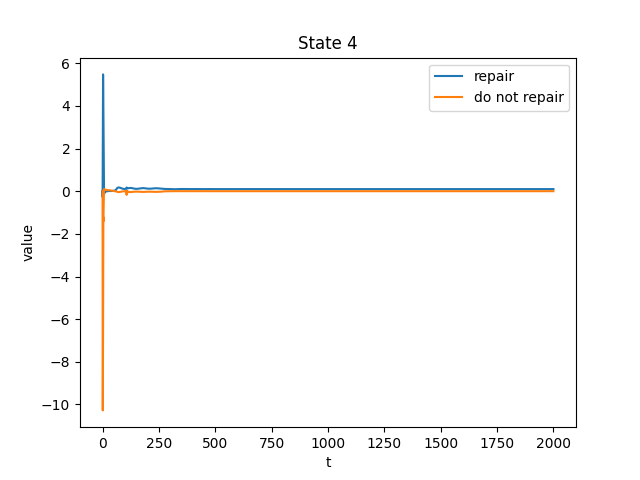}
		\caption{$\tau(4,a_1),\tau(4,a_2)$}
	\end{minipage}
 \begin{minipage}{0.31\linewidth}
        \centering
		\includegraphics[width=0.9\linewidth]{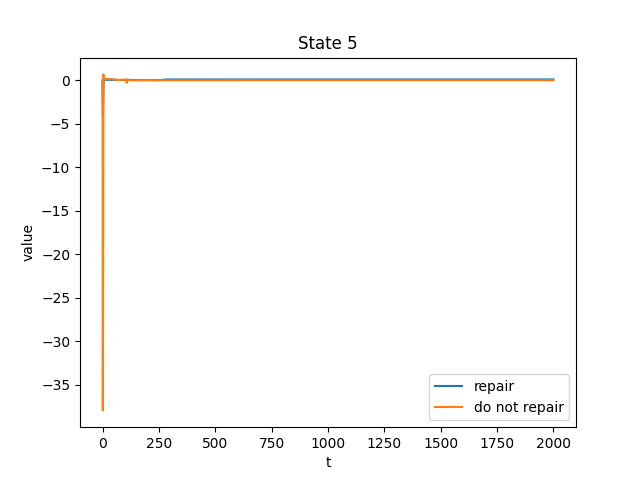}
		\caption{$\tau(5,a_1),\tau(5,a_2)$}
	\end{minipage}
\end{figure}

% \begin{subfigure}[h!]    
% \includegraphics[width=0.31 in]{state1_policy.png}
% \caption{$\tau(1,a_1),\tau(1,a_2)$}
% \end{subfigure}

% \subfloat[$\tau(2,a_1),\tau(2,a_2)$] { 
     
% \includegraphics[width=0.31\columnwidth]{state2_policy.png}     
% }    
% \subfloat[$\tau(3,a_1),\tau(3,a_2)$] { 
     
% \includegraphics[width=0.31\columnwidth]{state3_policy.png}     
% }   

% \subfloat[$\tau(4,a_1),\tau(4,a_2)$] { 
    
% \includegraphics[width=0.31\columnwidth]{state4_policy.png}     
% }   
% \subfloat[$\tau(5,a_1),\tau(5,a_2)$] { 
     
% \includegraphics[width=0.31\columnwidth]{state5_policy.png}     
% }   
% \caption{Optimal solution $\tau$ at the five states obtained by the DNN approach}     
% \label{fig}     

We also use the SCA algorithm to solve the problem \eqref{399}. Figure 5 shows the value of optimal objective function of problem \eqref{399} at each iteration of the SCA algorithm. Figure 6 shows the value of optimal objective function of the DNN approach with $t$ in $[50,2000]$. When $t$ is between $[0,50]$, the objective value varies intensely up to 7250 which is far from the convergence value, so we do not plot that part in the figure.
%there exists a huge gap of the value of objective function, which is up to nearly 2900. In order to better present the convergence performance, we choose the interval $[50,2000]$ for the DNN approach.

\begin{figure}[htbp]
	\centering
	\begin{minipage}{0.49\linewidth}
		\centering
		\includegraphics[width=0.9\linewidth]{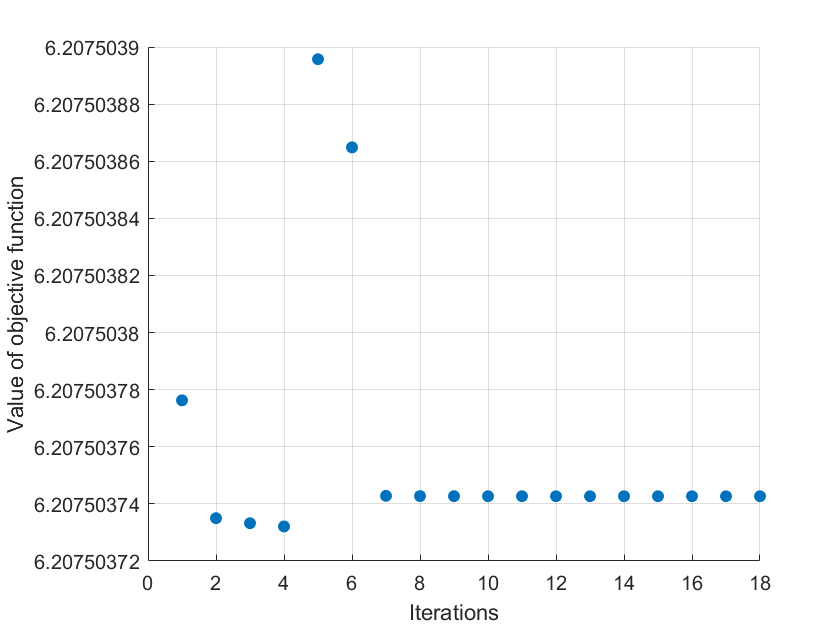}
		\caption{Objective value for SCA algorithm}
	\end{minipage}
	%\qquad
	\begin{minipage}{0.49\linewidth}
		\centering
		\includegraphics[width=0.9\linewidth]{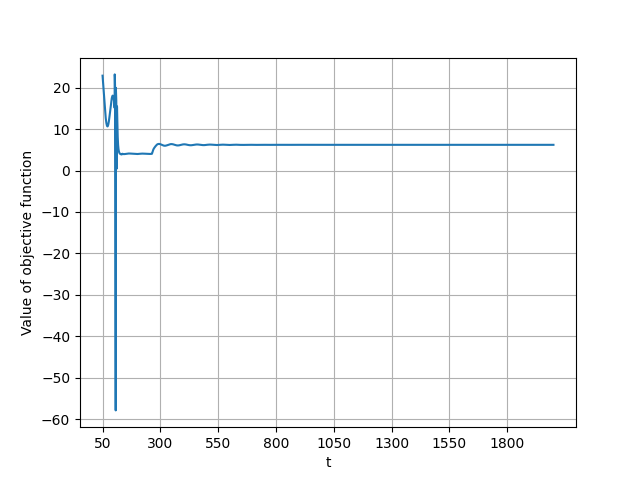}
		\caption{Objective value for DNN approach}
	\end{minipage}
\end{figure} 

Comparing Figure 5 and Figure 6, we find that a key feature of the DNN approach is time-continuity, while the SCA algorithm is with finite steps. The time-continuity of the DNN approach provides more information about how the objective function move towards the optimal value.  {We also see that the convergence speed of two approaches are both fast. For the SCA algorithm, the objective function converges within a relative small bound at the 7th iteration. For the DNN approach, the objective function converges within a relative small bound when $t=300$.}

\subsection{Accuracy of solutions}
We know from Theorem \ref{4} that the equilibrium point of the DNN model \eqref{ujm} is equivalent to the KKT point of \eqref{37}. Thanks to the Lyapunov stability of \eqref{ujm} by Theorem \ref{55}, the solution of DNN is guaranteed to converge to the equilibrium point when $t\rightarrow+\infty$. Based on the definition of the equilibrium point, the samller the value of $\Vert\frac{dz}{dt}\Vert,$ where $z=(\tau, x, \beta, \chi, \theta_1, \theta_2, \varphi)$, the closer the solution of \eqref{ujm} to the equilibrium point. Thus we introduce the following measure to evaluate the accuracy of the solution of our DNN model.
%which represents the gap between the solution of DNN and the KKT point of \eqref{37}.

\begin{definition}\label{DEF}
    The accuracy of a solution to the DNN model is the gap between the solution to the DNN model and the KKT point of \eqref{37}. We can compute the accuracy by finding the minimal $\epsilon\in\mathbb{R}$ such that
    $$\Vert\frac{d\tau}{dt}\Vert\le\epsilon, \Vert\frac{dx}{dt}\Vert\le\epsilon, \Vert\frac{d\beta}{dt}\Vert\le\epsilon, \Vert\frac{d\chi}{dt}\Vert\le\epsilon, \Vert\frac{d\zeta}{dt}\Vert\le\epsilon, \Vert\frac{d\theta_1}{dt}\Vert\le\epsilon, \Vert\frac{d\theta_2}{dt}\Vert\le\epsilon, \Vert\frac{d\varrho}{dt}\Vert\le\epsilon.$$
\end{definition}

Figure 7 shows the solution accuracy of the DNN model when solving the machine replacement problem in terms of Definition \ref{DEF}. For example, when $t=2000$, the accuracy $\epsilon=3.945292325991318e-07$. 
%As we find when $t$ is between $[0,250]$, there exists a huge gap of the accuracy value,
When $t\in [0,250]$, the accuracy value may reach nearly 90. Thus we only show the values in the interval $[250,4000]$ to observe the final convergence.

\begin{figure}[htbp]%调节图片位置，h：浮动；t：顶部；b:底部；p：当前位置
	\centering
	\includegraphics[width=0.6\linewidth]{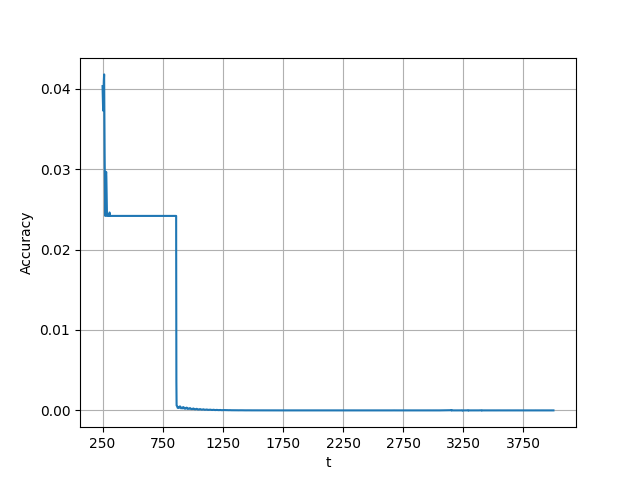}
	\caption{The accuracy of DNN model}
\end{figure}

%For the SCA algorithm, we can only choose the error bound $\tilde{\epsilon}$ between two separated problems in advance while we cannot depict the accuracy of its solution as for KKT conditions like DNN model. 
It can be observed from Figure 7 that the accuracy could be extremely low when $t\ge 1000$. However, for the SCA algorithm, we could not relate its accuracy of solutions with KKT conditions of the original problem. The only way to ensure its accuracy is to set a very low error bound $L$ in Algorithm 1 to force the solutions of two separated optimization problems as close as possible.

\subsection{Generalization performance}\label{num5}
To evaluate the generalization ability of the two solution methods, we examine the out-of-sample performance of the optimal solutions under some randomly generated distributions. We assume that the reference distribution is a conservative estimation of the true distribution, and randomly generate the mean vector of the simulated distribution related to the center of the moments-based ambiguity set. We use the “rand" function in MATLAB to generate 4 groups of out-of-sample distributions, which represent different levels of dispersion. In each group, we randomly generate 100 Gaussian distributions. 
In particular, in the first group, the randomly generated Gaussian distributions are $\mathcal{K}_k^j\sim N(\mu_{k}^{j},\Sigma_{k})$, $k=1,2, j=1,2,...,100$, where $\mu_{k}^{j}=\mu_{k}-1\cdot{\rm{rand}(1,10)}$, $k$ denotes the number of the reward vector and $j$ denotes the number of generation. In the second group, the randomly generated Gaussian distributions are $\mathcal{K}_k^j\sim N(\mu_{k}^{j},\Sigma_{k})$, where $\mu_{k}^{j}=\mu_{k}-2\cdot{\rm{rand}(1,10)}$. In the third group, the randomly generated Gaussian distributions are $\mathcal{K}_k^j\sim N(\mu_{k}^{j},\Sigma_{k})$, where $\mu_{k}^{j}=\mu_{k}-3\cdot{\rm{rand}(1,10)}$. In the forth group, the randomly generated Gaussian distributions are $\mathcal{K}_k^j\sim N(\mu_{k}^{j},\Sigma_{k})$, where $\mu_{k}^{j}=\mu_{k}-3.5\cdot{\rm{rand}(1,10)}$.

With the above setting, we use 100 randomly generated distributions $\mathcal{K}_{k}^{j}, j=1,2,...,100$ in each group and the optimal solutions $\tau$ obtained by the DNN approach and the SCA algorithm to compute the satisfaction probability $$\mathbb{P}_{\mathcal{K}^{j}}(\tau^{\top} r_k\ge \xi_{k}, k=1,2,...,K)=\prod\limits_{k=1}^{K}\mathbb{P}_{\mathcal{K}_k^{j}}(\tau^{\top}\cdot r_k\ge \xi_{k})=\prod\limits_{k=1}^{K}\left(1-\mathcal{H}^{-1}(\frac{\xi_k-\tau^{\top}\mu_k^{j}}{\Vert\tau^{\top}\Sigma_{k}\tau\Vert_{2}})\right),$$ where $\mathcal{H}$ is the standard Gaussian cumulative distribution function. The concrete results are depicted in Figure 8.

\begin{figure}[htbp]
	\centering
	\begin{minipage}{0.49\linewidth}
		\centering
		\includegraphics[width=0.9\linewidth]{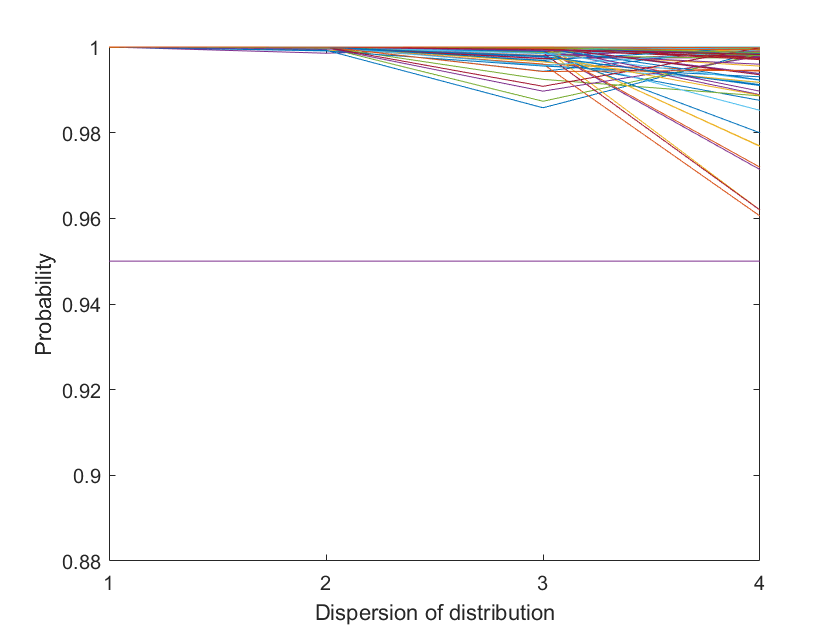}
	\end{minipage}
	%\qquad
	\begin{minipage}{0.49\linewidth}
		\centering
		\includegraphics[width=0.9\linewidth]{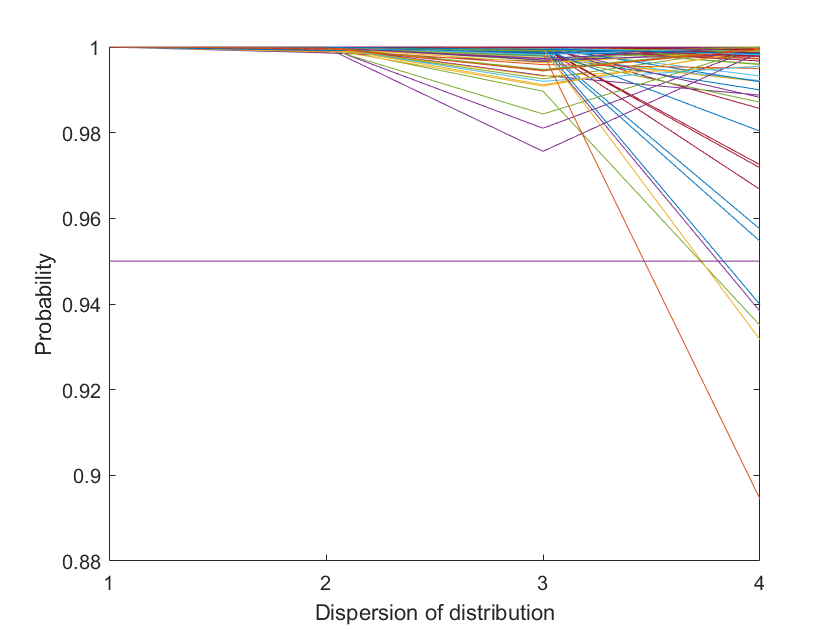}
	\end{minipage}
 \caption{Values of $\mathbb{P}_{\mathcal{K}^{j}}(\tau^{\top} r_k\ge \xi_{k}, k=1,2,...,K), j=1,2,...,100$ with randomly generated distributions $\mathcal{K}^{j}$, $j=1,2,...,100$, where the optimal solutions $\tau$ are obtained by DNN approach and SCA algorithm, respectively.} 
\end{figure}

% \begin{figure}[h!] \centering    
% \subfloat[DNN] {
%  \label{fig:a}     
% \includegraphics[width=0.47\columnwidth]{DNN_robust2.png}  
% }     
% \subfloat[SCA] { 
     
% \includegraphics[width=0.47\columnwidth]{Sequential_robust.png}     
% }       
% \caption{Values of $\mathbb{P}_{\mathcal{K}^{j}}(\tau^{\top} r_k\ge \xi_{k}, k=1,2,...,K), j=1,2,...,100$ with randomly generated distributions $\mathcal{K}^{j}$, $j=1,2,...,100$, where the optimal solutions $\tau$ are obtained by DNN approach and SCA algorithm, respectively.}          
% \end{figure}

It can be found from Figure 8 that there are 5 out-of-sample distributions in the forth test group, under which the solution found by SCA algorithm fails to guarantee the satisfaction of the chance constraint, while the solution obtained by DNN approach can guarantee the satisfaction in all of the 4 groups. Furthermore, comparing their out-of-sample performance, we can observe that the solution obtained by the DNN approach in complex situations, i.e., 100 randomly generated distributions, is more robust than that derived by the SCA algorithm.

\subsection{Conclusion of numerical experiments}\label{cne}
Based on a series of numerical experiments above, our conclusions of the DNN model in solving the moment-based J-DRCCMDP problem compared with the SCA algorithm are as follows.

 {Firstly, we can find by Table 2 that the numerical solutions of DNN approach and SCA algorithm are nearly the same in this problem, and they both consistent to the aging effect to the machine replacement. Secondly, through convergence results, we can see the objective values obtained by DNN approach change continuously.
%changing trend of . 
However the objective values obtained by SCA algorithm are dispersed, and we cannot capture the whole changing trend. Thirdly, with the %advantages of 
Lyapunov stability in Theorem \ref{55} and the equivalence between the equilibrium point and the KKT point in Theorem \ref{4}, 
the solutions obtained by DNN approach are surely converging to an equilibrium point which corresponds to a KKT point of the original optimization problem.
%point of the corresponding ODE system as $t$ increases. 
Therefore we can depict the accuracy of DNN approach in correspondence to KKT conditions. While for the SCA algorithm, the level of accuracy cannot be reduced continuously w.r.t. KKT conditions.
%of the original optimization problem. 
Lastly, in out-of-sample experiments, we compare the solutions obtained by two approaches. We find that the solution obtained by DNN approach in complex simulation environment performs more robustly
%better quality of robustness 
than the solution obtained by SCA algorithm, which indicates that the generalization ability of solutions obtained by DNN approach are better than those obtained by SCA algorithm.}

\section{Conclusion}
In this paper, we study the moment-based joint DRCCMDP and derive its deterministic reformulation. We propose a dynamical neural network approach to solve the resulting nonconvex optimization problem. 
%We also apply the Sequential convex approximation algorithm to solve it. 
In the numerical experiment, we examine 
%study on four parts aspects of results to discuss the quality of DNN approach. As a comparison, we state our findings about
the pros and cons of the DNN approach and the SCA algorithm in a machine replacement problem.

For the future research, %one can consider 
using the DNN approach to solve the DRCCMDP with other kinds of ambiguity sets or studying the DRCCMDP with joint ambiguity on %oth 
the distribution of rewards and the transition probabilities %, which 
are promising topics.

\begin{acknowledgements}
This research was supported by National Natural Science Foundation of China (No. 11991023 and 11901449) and National Key R\&D Program of China (No. 2022YFA1004000) 
\end{acknowledgements}


\begin{thebibliography}{plain}

\bibitem{varagapriya2022constrained}
Varagapriya V, Singh V~V, and Lisser A.
\newblock Constrained markov decision processes with uncertain costs.
\newblock {\em Operations Research Letters}, 50(2):218--223, 2022.

\bibitem{mannor2016robust}
Mannor S, Mebel O, and Xu~H.
\newblock Robust mdps with k-rectangular uncertainty.
\newblock {\em Mathematics of Operations Research}, 41(4):1484--1509, 2016.

\bibitem{ramani2022robust}
Ramani S and Ghate A.
\newblock Robust markov decision processes with data-driven, distance-based
  ambiguity sets.
\newblock {\em SIAM Journal on Optimization}, 32(2):989--1017, 2022.

\bibitem{brechtel2014probabilistic}
Brechtel S, Gindele T, and Dillmann R.
\newblock Probabilistic decision-making under uncertainty for autonomous
  driving using continuous pomdps.
\newblock In {\em 17th international IEEE conference on intelligent
  transportation systems (ITSC)}, pages 392--399. IEEE, 2014.

\bibitem{liu2023continual}
Liu S, Wang B, Li~H, and et~al.
\newblock Continual portfolio selection in dynamic environments via incremental
  reinforcement learning.
\newblock {\em International Journal of Machine Learning and Cybernetics},
  14(1):269--279, 2023.

\bibitem{klabjan2013robust}
Klabjan D, Simchi-Levi D, and Song M.
\newblock Robust stochastic lot-sizing by means of histograms.
\newblock {\em Production and Operations Management}, 22(3):691--710, 2013.

\bibitem{wang2020mdp}
Wang C, Lei S, Ju~P, and et~al.
\newblock M\uppercase{DP}-based distribution network reconfiguration with
  renewable distributed generation: Approximate dynamic programming approach.
\newblock {\em IEEE Transactions on Smart Grid}, 11(4):3620--3631, 2020.

\bibitem{goyal2022robust}
Goyal V and Grand-Clement J.
\newblock Robust markov decision processes: Beyond rectangularity.
\newblock {\em Mathematics of Operations Research}, 2022.

\bibitem{you2019advanced}
You C~X, Lu~J~B, Filev D, and et~al.
\newblock Advanced planning for autonomous vehicles using reinforcement
  learning and deep inverse reinforcement learning.
\newblock {\em Robotics and Autonomous Systems}, 114:1--18, 2019.

\bibitem{delage2010percentile}
Delage E and Mannor S.
\newblock Percentile optimization for markov decision processes with parameter
  uncertainty.
\newblock {\em Operations Research}, 58(1):203--213, 2010.

\bibitem{yu2022zero}
Yu~Z~H, Guo X~P, and Xia L.
\newblock Zero-sum semi-markov games with state-action-dependent discount
  factors.
\newblock {\em Discrete Event Dynamic Systems}, 32(4):545--571, 2022.

\bibitem{ma2019state}
Ma~S and Yu~J~Y.
\newblock State-augmentation transformations for risk-sensitive reinforcement
  learning.
\newblock In {\em Proceedings of the AAAI Conference on Artificial
  Intelligence}, volume~33, pages 4512--4519, 2019.

\bibitem{prashanth2014policy}
Prashanth L~A.
\newblock Policy gradients for \uppercase{CV}a\uppercase{R}-constrained
  \uppercase{MDP}s.
\newblock In {\em International Conference on Algorithmic Learning Theory},
  pages 155--169. Springer, 2014.

\bibitem{kuccukyavuz2022chance}
K{\"u}{\c{c}}{\"u}kyavuz S and Jiang R~W.
\newblock Chance-constrained optimization under limited distributional
  information: a review of reformulations based on sampling and distributional
  robustness.
\newblock {\em EURO Journal on Computational Optimization}, page 100030, 2022.

\bibitem{varagapriya2022joint}
Varagapriya V, Singh V~V, and Lisser A.
\newblock Joint chance-constrained markov decision processes.
\newblock {\em Annals of Operations Research}, pages 1--23, 2022.

\bibitem{wiesemann2014distributionally}
Wiesemann W, Kuhn D, and Sim M.
\newblock Distributionally robust convex optimization.
\newblock {\em Operations Research}, 62(6):1358--1376, 2014.

\bibitem{delage2010distributionally}
Delage E and Ye~Y~Y.
\newblock Distributionally robust optimization under moment uncertainty with
  application to data-driven problems.
\newblock {\em Operations Research}, 58(3):595--612, 2010.

\bibitem{rahimian2019distributionally}
Hamed R~M and Sanjay M.
\newblock Distributionally robust optimization: A review.
\newblock {\em arXiv preprint arXiv:1908.05659}, 2019.

\bibitem{hu2013kullback}
Hu~Z~L and Hong L~J.
\newblock Kullback-leibler divergence constrained distributionally robust
  optimization.
\newblock {\em Available at Optimization Online}, pages 1695--1724, 2013.

\bibitem{gao2022distributionally}
Gao R and Kleywegt A.
\newblock Distributionally robust stochastic optimization with wasserstein
  distance.
\newblock {\em Mathematics of Operations Research}, 2022.

\bibitem{xie2021distributionally}
Xie W~J.
\newblock On distributionally robust chance constrained programs with
  wasserstein distance.
\newblock {\em Mathematical Programming}, 186(1):115--155, 2021.

\bibitem{nguyen2022distributionally}
Nguyen H~N, Lisser A, and Singh V~V.
\newblock Distributionally robust chance-constrained markov decision processes.
\newblock {\em arXiv preprint arXiv:2212.08126}, 2022.

\bibitem{wright1997primal}
Wright S~J.
\newblock {\em Primal-dual Interior-Point Methods}.
\newblock SIAM, 1997.

\bibitem{nelder1965simplex}
Nelder J~A and Mead R.
\newblock A simplex method for function minimization.
\newblock {\em The computer journal}, 7(4):308--313, 1965.

\bibitem{hopfield1985neural}
Hopfield J~J and Tank D~W.
\newblock “\uppercase{N}eural” computation of decisions in optimization
  problems.
\newblock {\em Biological cybernetics}, 52(3):141--152, 1985.

\bibitem{wang1993analysis}
Wang J.
\newblock Analysis and design of a recurrent neural network for linear
  programming.
\newblock {\em IEEE Transactions on Circuits and Systems I: Fundamental Theory
  and Applications}, 40(9):613--618, 1993.

\bibitem{xia1996new}
Xia Y~S.
\newblock A new neural network for solving linear programming problems and its
  application.
\newblock {\em IEEE Transactions on Neural Networks}, 7(2):525--529, 1996.

\bibitem{ko2011recurrent}
Ko~C~H, Chen J~S, and Yang C~Y.
\newblock Recurrent neural networks for solving second-order cone programs.
\newblock {\em Neurocomputing}, 74(17):3646--3653, 2011.

\bibitem{nazemi2020new}
Nazemi A and Sabeghi A.
\newblock A new neural network framework for solving convex second-order cone
  constrained variational inequality problems with an application in
  multi-finger robot hands.
\newblock {\em Journal of Experimental \& Theoretical Artificial Intelligence},
  32(2):181--203, 2020.

\bibitem{xia1996new1}
Xia~Y S.
\newblock A new neural network for solving linear and quadratic programming
  problems.
\newblock {\em IEEE transactions on neural networks}, 7(6):1544--1548, 1996.

\bibitem{nazemi2014neural}
Nazemi A.
\newblock A neural network model for solving convex quadratic programming
  problems with some applications.
\newblock {\em Engineering Applications of Artificial Intelligence}, 32:54--62,
  2014.

\bibitem{feizi2021solving}
Feizi A, Nazemi A, and Rabiei M~R.
\newblock Solving the stochastic support vector regression with probabilistic
  constraints by a high-performance neural network model.
\newblock {\em Engineering with Computers}, pages 1--16, 2021.

\bibitem{forti2004generalized}
Forti M, Nistri P, and Quincampoix M.
\newblock Generalized neural network for nonsmooth nonlinear programming
  problems.
\newblock {\em IEEE Transactions on Circuits and Systems I: Regular Papers},
  51(9):1741--1754, 2004.

\bibitem{wu1996high}
Wu~X~Y, Xia Y~S, Li~J~M, and et~al.
\newblock A high-performance neural network for solving linear and quadratic
  programming problems.
\newblock {\em IEEE transactions on neural networks}, 7(3):643--651, 1996.

\bibitem{nazemi2011dynamical}
Nazemi A.
\newblock A dynamical model for solving degenerate quadratic minimax problems
  with constraints.
\newblock {\em Journal of Computational and Applied Mathematics},
  236(6):1282--1295, 2011.

\bibitem{gao2004neural}
Gao X~B, Liao L~Z, and Xue W~M.
\newblock A neural network for a class of convex quadratic minimax problems
  with constraints.
\newblock {\em IEEE Transactions on Neural Networks}, 15(3):622--628, 2004.

\bibitem{wu2022dynamical}
Wu~D~W and Lisser A.
\newblock A dynamical neural network approach for solving stochastic two-player
  zero-sum games.
\newblock {\em Neural Networks}, 152:140--149, 2022.

\bibitem{tassouli2023neural}
Tassouli S and Lisser A.
\newblock A neural network approach to solve geometric programs with joint
  probabilistic constraints.
\newblock {\em Mathematics and Computers in Simulation}, 205:765--777, 2023.

\bibitem{dissanayake1994neural}
Dissanayake M and Phan-Thien N.
\newblock Neural-network-based approximations for solving partial differential
  equations.
\newblock {\em Communications in Numerical Methods in Engineering},
  10(3):195--201, 1994.

\bibitem{lagaris1998artificial}
Lagaris I~E, Likas A, and Fotiadis D.
\newblock Artificial neural networks for solving ordinary and partial
  differential equations.
\newblock {\em IEEE transactions on neural networks}, 9(5):987--1000, 1998.

\bibitem{flamant2020solving}
Flamant C, Protopapas P, and Sondak D.
\newblock Solving differential equations using neural network solution bundles.
\newblock {\em arXiv preprint arXiv:2006.14372}, 2020.

\bibitem{dong2023adversarial}
Dong H~Y, Dong J~L, Yuan S, and et~al.
\newblock Adversarial attack and defense on natural language processing in deep
  learning: A survey and perspective.
\newblock In {\em International Conference on Machine Learning for Cyber
  Security}, pages 409--424. Springer, 2023.

\bibitem{amiri2023novel}
Amiri M, Jafari A~H, Makkiabadi B, and et~al.
\newblock A novel un-supervised burst time dependent plasticity learning
  approach for biologically pattern recognition networks.
\newblock {\em Information Sciences}, 622:1--15, 2023.

\bibitem{hong2011sequential}
Hong L~J, Yang Y, and Zhang L~W.
\newblock Sequential convex approximations to joint chance constrained
  programs: A monte carlo approach.
\newblock {\em Operations Research}, 59(3):617--630, 2011.

\bibitem{liu2016stochastic}
Liu J, Lisser A, and Chen Z~P.
\newblock Stochastic geometric optimization with joint probabilistic
  constraints.
\newblock {\em Operations Research Letters}, 44(5):687--691, 2016.

\bibitem{sutton1999policy}
Sutton R~S, McAllester D, Singh S, and et~al.
\newblock Policy gradient methods for reinforcement learning with function
  approximation.
\newblock {\em Advances in neural information processing systems}, 12, 1999.

\bibitem{altman1999constrained}
Altman E.
\newblock {\em Constrained Markov Decision Processes: Stochastic Modeling}.
\newblock Routledge, 1999.

\bibitem{liu2022distributionally}
Liu J, Lisser A, and Chen Z~P.
\newblock Distributionally robust chance constrained geometric optimization.
\newblock {\em Mathematics of Operations Research}, 47(4):2950--2988, 2022.

\bibitem{shapiro2021lectures}
Shapiro A, Dentcheva D, and Ruszczynski A.
\newblock {\em Lectures on stochastic programming: modeling and theory}.
\newblock SIAM, 2021.

\bibitem{bartlett1946statistical}
Bartlett M~S and Kendall D~G.
\newblock The statistical analysis of variance-heterogeneity and the
  logarithmic transformation.
\newblock {\em Supplement to the Journal of the Royal Statistical Society},
  8(1):128--138, 1946.

\bibitem{gorski2007biconvex}
Gorski J, Pfeuffer F, and Klamroth K.
\newblock Biconvex sets and optimization with biconvex functions: a survey and
  extensions.
\newblock {\em Mathematical Methods of Operations Research}, 66(3):373--407,
  2007.

\bibitem{jiang2021partial}
Jiang M, Meng Z~Q, and Shen R.
\newblock Partial exactness for the penalty function of biconvex programming.
\newblock {\em Entropy}, 23(2):132, 2021.

\bibitem{hu2005theory}
Hu~J~S and Li~W~P.
\newblock Theory of \uppercase{O}rdinary \uppercase{D}ifferential
  \uppercase{E}quations. \uppercase{E}xistence, \uppercase{U}niqueness and
  \uppercase{S}tability.
\newblock {\em Publications of Department of Mathematics. The Hong Kong
  University of Science and Technology}, 2005.

\bibitem{foias1990positive}
Foias C and Frazho A~E.
\newblock Positive definite block matrices.
\newblock In {\em The Commutant Lifting Approach to Interpolation Problems},
  pages 547--586. Springer, 1990.

\bibitem{rockafellar2009variational}
Rockafellar R~T and Wets R~J~B.
\newblock {\em Variational Analysis}, volume 317.
\newblock Springer Science \& Business Media, 2009.

\bibitem{slotine1991applied}
Slotine J~J~E, Li~W~P, and et~al.
\newblock {\em Applied Nonlinear Control}, volume 199.
\newblock Prentice hall Englewood Cliffs, NJ, 1991.

\bibitem{wiesemann2013robust}
Wiesemann W, Kuhn D, and Rustem B.
\newblock Robust markov decision processes.
\newblock {\em Mathematics of Operations Research}, 38(1):153--183, 2013.

\end{thebibliography}
\end{document}